\documentclass[10pt, a4paper, twosided]{article}
\usepackage{amssymb,amsmath}

\usepackage[utf8]{inputenc}

\usepackage{tikz,inputenc,xcolor,amsthm,hyperref, cite, color, url,amsfonts,amsmath,amssymb,caption,multicol,graphicx,fancyhdr,tikz,pgfplots,subcaption,comment,commath,bm}

\usepgfplotslibrary{groupplots}

\usepackage[toc,page]{appendix}

\usepackage{bbm}

\usepackage{dsfont}

\usepackage{graphicx}
\graphicspath{{images/}}

\usepackage{caption}

\usepackage{geometry}
\usepackage{algorithm}
\usepackage[noend]{algpseudocode}
\usepackage{fullwidth}
\usepackage{lipsum}
\usepackage{pifont}

\usepackage{todonotes}

\renewcommand{\d}{\delta}
\renewcommand{\l}{\lambda}
\newcommand{\p}{\partial}

\newcommand{\g}{\gamma}

\newcommand{\e}{\varepsilon}

\newcommand{\f}[2]{\frac{#1}{#2}}

\newtheorem{theorem}{Theorem}[section]
\newtheorem{lemma}[theorem]{Lemma}
\newtheorem{proposition}[theorem]{Proposition}
\newtheorem{property}[theorem]{Property}

\theoremstyle{definition}

\newtheorem{remark}[theorem]{Remark}

\usepackage[affil-it]{authblk}
\author[a]{M. Saxena\footnote{m.mayank@tue.nl}}
\author[b]{I. Dimitriou \footnote{ idimit@math.upatras.gr}}
\author[a]{S. Kapodistria \footnote{s.kapodistria@tue.nl}}
\affil[a]{\small Department of Mathematics and Computer Science, Eindhoven University of Technology, P.O.~Box 513, 5600MB, Eindhoven.}
\affil[b]{\small Department of Mathematics, 
	University of Patras, P.O.~Box 
	26500, Patras, Greece.}
\date{\small \today}

\allowdisplaybreaks

\usepackage[numbers]{natbib}

\begin{document}
	\title{Analysis of the shortest relay queue policy in a \\ cooperative random access network with collisions}

	\maketitle
	
	
	\begin{abstract} 
		The scope of this work is twofold: On the one hand, strongly motivated by emerging engineering issues in multiple access communication systems, we investigate the performance of a slotted-time relay-assisted cooperative random access wireless network with collisions and with join the shortest queue relay-routing protocol. For this model, we investigate the stability condition, and apply different methods to derive the joint equilibrium distribution of the queue lengths. 
		On the other hand, using the cooperative communication system as a vehicle for illustration, we investigate and compare three different approaches for this type of multi-dimensional stochastic processes, namely the compensation approach, the power series algorithm (PSA), and the probability generating function (PGF) approach. 
		We present an extensive numerical comparison of the compensation approach and PSA, and discuss which method performs better in terms of accuracy and computation time. We also  provide details on how to compute the PGF in terms of a solution of a Riemann-Hilbert boundary value problem.  
	\end{abstract}
	\vspace{2mm}
	

	\section{Introduction}
	Cooperative communication is a new communication paradigm in which different terminals (i.e., nodes, devices) in a wireless network share their antennas and resources for distributed transmission and processing. Recent studies have shown that cooperative communications yield significant performance improvements for 5G networks, which need massive uncoordinated access, low latency, energy efficiency and ultra-reliability \cite{nos}.
	
	The unprecedented growth of wireless networking, and the ever growing demand for higher data rates and capacity over the last decades, have already pushed the limits of current cellular systems \cite{teh}. By exploiting the spatial diversity inherent to wireless channels, which is an important tool to overcome the effects of fading (decrease in signal power due to path loss), shadowing and attenuation (decrease in signal strength), relay-based cooperative communications have been proposed as the appropriate solution to achieve the requirements of future needs; see e.g., \cite{hong,liu}.

	A typical relay-based cooperative wireless network operates as follows: There exists a network of a finite number of source users, a finite number of relay nodes and a common destination node. The source users transmit packets to the destination node with the cooperation of the relays. If a direct transmission of a user's packet to the destination fails, a cooperation strategy among sources and relays is employed to specify the relays that will store the blocked packet in their buffer. Relays are responsible for the transmission of the blocked packets to the destination, e.g., \cite{,pap2}. In a wireless network, transmission failures occur either due to packet collisions, or due to channel fading/noise and attenuation. In both cases the  packet has to be re-transmitted at a later slot. The former case occurs when more than one node transmit simultaneously, while the latter one refers to the probabilistic nature of transmissions, see e.g., \cite{pap2,rong,sad}. 
	
	In this work, we consider a simple relay-based cooperative wireless network with a single source, two infinite capacity relay nodes, and a common destination with collisions under a load balancing relay scheme. We assume that due to deep fading and bad channel quality it is impossible for the source to communicate with the destination through a direct link, and thus, cooperation within the relays is imperative. Furthermore, we assume that the relays and the source user are sufficiently close such that the packets transmitted over the channel are always correctly received by the relays. The employed cooperation strategy among source and relays is queue-based, with as ultimate goal to minimise the total transmission time, i.e., the time that is needed to transmit a packet from the source to the destination. To this end, we consider join the shortest queue policy as it seems to be the most appropriate for such a wireless network \cite{magu,mukherjee,borst,sri1}.

	\subsection{Related work}
	
	For networks with cooperation the first point of interest concerns the characterisation of the stable throughput region, aka the stability region. For small, simple networks the stability region can be fully determined, see, e.g., \cite{tsy,dim0,naw}, while for large, general networks, only bounds of these regions are known \cite{sem,raoephre,luo,szpa}. For a thorough overview of several techniques used for the derivation of the stability region the interested reader is referred to  \cite{foss2004overview}. An alternative way to derive stability conditions is to use the concept of dominant systems, under which the network of interest is (stochastically) compared to a simpler one that is easier to analyse, see, e.g., \cite{pap3} and the references therein. Another powerful tool to investigate necessary and sufficient stability conditions for work-conserving queueing networks is the use of fluid models, see \cite{stol2,seva,foss1,foss2,ghad}.
	
	Next to the characterisation of the stability region, the delay performance (i.e., the investigation of the joint relay queue length distribution) is another crucial performance measure in random access networks, which recently has regained attention due to the rapid development of real-time applications that require delay-based guarantees \cite{gao}. However, the impact of interacting queues causes severe mathematical difficulties, and there are very few studies that deal with the stability of the queueing delay. In \cite{naka}, by performing an appropriate truncation of the infinite Markov chain, the authors approximated the steady-state probability vector, within any desired degree of accuracy, whereas in \cite{tak}, diffusion approximations were applied. In \cite{nain,dim0, dim3} the PGF of the joint (relay) queue length distribution was obtained in terms of the solution of a boundary value problems. In \cite{WangTong2010,ProutiereITC2011} fluid models were used to investigate the delay analysis of random access networks. In \cite{szpa0} bounds for the queueing delay in a random access network with $N>2$ nodes were also derived. 
	
	Note that, in the vast majority of the above mentioned literature, each user node in the network has its own dedicated traffic. Alternatively, the inclusion of cooperation under certain criteria  gives rise to the introduction of load balancing techniques that forward the packets to specific relays with ultimate goal the optimisation of the overall network performance; e.g., \cite{liu}.

	Load balancing schemes are used to improve scalability and overall system throughput in distributed systems. Such schemes improve the system's performance by dividing the work effectively among the participating nodes. Under certain conditions (identical servers, exponential service times, and service protocols that do not depend on the service requirement, such as the FCFS), the so-called join the shortest queue (JSQ) policy has several strong optimality properties: The JSQ policy minimises the overall mean delay among the class of load balancing policies that do not have any advance knowledge of the service requirements \cite{borst,ephre,mukherjee}.

	For small scale systems with two queues, the JSQ policy has been extensively studied. The JSQ policy was introduced in \cite{haight}, in the context of two parallel (exponential) servers and a Poisson stream of arrivals. The first major steps towards its exact analysis were made in \cite{king0,flat}, using a uniformization approach that established that the PGF of the joint equilibrium distribution of the queue lengths is meromorphic (i.e., the equilibrium distribution can be written as a linear combination - finite or infinite - of product-form terms). For an extensive treatment of the JSQ system using a generating function approach, the interested reader is referred to  \cite{cohenboxma,fayolle}. An alternative approach that is not based on generating functions is the compensation approach, that directly solves the balance equations and obtains the equilibrium distribution in the form of a linear combination - finite or infinite - of product-form terms, see \cite{ad0,ad1}. It is worth noting, that although there exist multiple analyses for the JSQ system, this line of work is very challenging, see \cite{Onno} for a comparison of some of the approaches and an exposition of their strengths and limitations. Using the exact analysis for the case of two servers, Gupta et al. \cite{Mor} obtained the first approximate analysis of the marginal equilibrium distribution of the queue lengths in a server farm with JSQ routing protocol, processor sharing service discipline and generally distributed service times.
	
	A mainly numerical approach, that has been successfully used for among others JSQ systems, is the power-series algorithm (PSA), \cite{blanc1987,blanc1992}, which assumes that the equilibrium distribution can be written as a function of a system parameter (such as the load). 
	In a different direction, the precedence relation (PR) method, see, e.g., \cite{Houtum}, can be used to systematically construct bounds for any performance measure of the Markov chain under consideration. 
	The strength of the PSA and the PR methods is that they are not restricted to two queues, but apply equally well to more than two queues. A weak point, however, is that the theoretical foundation of PSA is still incomplete and that PR produces no exact results, but bounds.

	For large scale systems, the exact analysis is elusive and the vast majority of the literature focuses on asymptotic regimes, and heavy-traffic delay optimality in steady state. 
	The most common asymptotic regimes include fluid limits, heavy traffic limits, mean field limits and the Halfin-Whitt (quality-efficiency-driven) regime. Under the assumption of a Poisson arrival stream with rate $\lambda$, $N$ identical (exponential) servers, each with service rate $\mu$, the above regimes can be viewed as different limits of the quantity $\lambda-N\mu$, after appropriately rescaling by some function of $N$. See the seminal paper of Foschini and Salz \cite{Foschini} for the case $N=2$, and \cite{Badonnel} for the mean field limit analysis of a large underloaded parallel server model.


	\subsection{Contribution}
	\paragraph{Application oriented contribution.} In this work, we consider a slotted-time relay-assisted cooperative communication system with a JSQ routing protocol at the relays and collisions. Such models have not been extensively analysed and very little is known regarding their delay performance. In particular, we provide insights on the characterisation of the delay and the performance of the system at hand. This is a particularly difficult task even in small scale systems, due to the strong interdependence/interaction between the queues. Furthermore, we demonstrate how to obtain the equilibrium distribution of the joint queue lengths and present an extensive numerical comparison of the techniques implemented, discussing which performs better in terms of accuracy and time complexity (computational time).

	\paragraph{Fundamental contribution.} In this work, we extend for the first time the framework of the compensation approach to a new class of random walks, viz., that violates one of the main assumptions of the compensation approach that of transitions being allowed only to neighbouring states. We show, using the system at hand as a vehicle for illustration, that the compensation approach can be extended to random walks in the quadrant that obey the following conditions:
	\begin{itemize}
		\item Homogeneity: The same transitions occur according to the same rates for all interior points, and similarly for all points on the horizontal boundary, and for all points on the vertical boundary.
		\item Forbidden steps: No transitions from interior states to the North, North-East, and East.
		\item  Bounded transitions: Only transitions to a bounded region.
	\end{itemize}
	Note that the compensation approach was developed to satisfy a more restrictive setting than the {\em bounded transitions}, by allowing only transition to the nearest neighbours, see, e.g., \cite{ad0,ad1,ad5}. This is a very promising result, since it seems to be possible to extend the compensation approach to random walks with large steps (extending the results obtained in \cite{adan2013}) and to queueing systems with (bounded) batch arrivals and/or (bounded) batch departures.

	\subsection{Paper structure}
	The paper is structured as follows: In Section \ref{Themodel}, we describe the system under consideration in detail, and provide its stability condition. The computation of the equilibrium distribution of the joint queue lengths using the compensation approach is performed in Section \ref{Sec4}.  In Section \ref{Sec:PSA}, we apply  PSA, while in Section \ref{gen}, we provide a detailed analysis on how to derive the probability generating function of the equilibrium joint queue length distribution in terms of a solution of a Riemann-Hilbert boundary value problem. Section \ref{Sec:num} is devoted to the comparison of the compensation approach and  PSA method. Finally, we present conclusions and possible generalisations in Section \ref{Sec:con}.

	\section{The model}\label{Themodel}
	We consider a relay-assisted cooperative random access wireless network composed of a saturated source user, that transmits packets to a common destination node, under the cooperation of two relay nodes. The relays are equipped with infinite capacity buffers (queues), and they assist the user by transmitting the packets that failed to reach the destination. Packets have equal length (i.e., they consist of the same fixed number of bits \cite{chan}), and time is divided into slots corresponding to the transmission time of a packet. We consider an Early Arrival System (EAS), under which at the beginning of a slot packets arrive and they are routed to the relays according to the join the shortest relay queue (JSRQ) policy. On the other hand, departures are scheduled at the end of the slot. 
	
	Packet arrivals are assumed i.i.d. Bernoulli random variables from slot to slot, with the average number of arrivals being $\lambda$ packets per slot. Upon the arrival of a packet, the source and the relays cooperate as follows: When the source transmits a packet, it is forwarded to the least loaded relay, i.e., to the relay with the smallest number of backlogged packets. Then, the relay node sends an acknowledgement to the source and takes over the responsibility
	of delivering the packet to the destination node by storing it in its queue. Such a protocol helps to keep a fair balance among the relays, as well as, it enhances the energy conservation of the relays (the relay node is usually a battery operated wireless device). In case the numbers of packets in the relay queues are equal, a packet is routed to relay $r$ with probability $\pi_{r}$, $r=1,2$. At the end of each slot, relay $r$ (if it is non-empty) transmits a packet to the destination node with probability $a_{r}$, $r=1,2$. If both relays transmit at the same slot, a collision occurs, and both packets have to be retransmitted in a later slot. If only one relay transmits, then the destination node successfully decodes it, sends an acknowledgement to the corresponding relay, and the packet exits the network. The acknowledgements are assumed to be error-free, instantaneous, and broadcasted to all relevant nodes. The nodes remove the successfully transmitted packets from their queues, while unsuccessful packets are retained. 
	
	In order to enhance the readability of the paper, and only for this reason, we focus hereon at the symmetric system, under which $a_{r}=a$, $r=1,2$, i.e., both relays have identical transmission parameters, and $\pi_{r}=1/2$, $r=1,2$. Note that the analysis that follows can be directly generalised to the asymmetric system case, however this would render the notation more complicated, which would severely impact the readability of the paper.

	Let $Q_{r}(n)$ be the number of stored packets at the buffer of relay $r$, $r=1,2$, at the beginning of the $n$-th slot, $n \geq 0$. Then $\{\bm{Q}(n),\, n \geq 0\}:=\{(Q_{1}(n),Q_{2}(n)),\, n \geq 0\}$ is a discrete time Markov chain with state space $\mathcal{S}=\{(i,j):\, i,j\geq0\}$. The corresponding probability transition diagram is depicted in Figure \ref{rw}. 
	\begin{figure}[h!]
		\centering	\includegraphics[scale= 0.9]{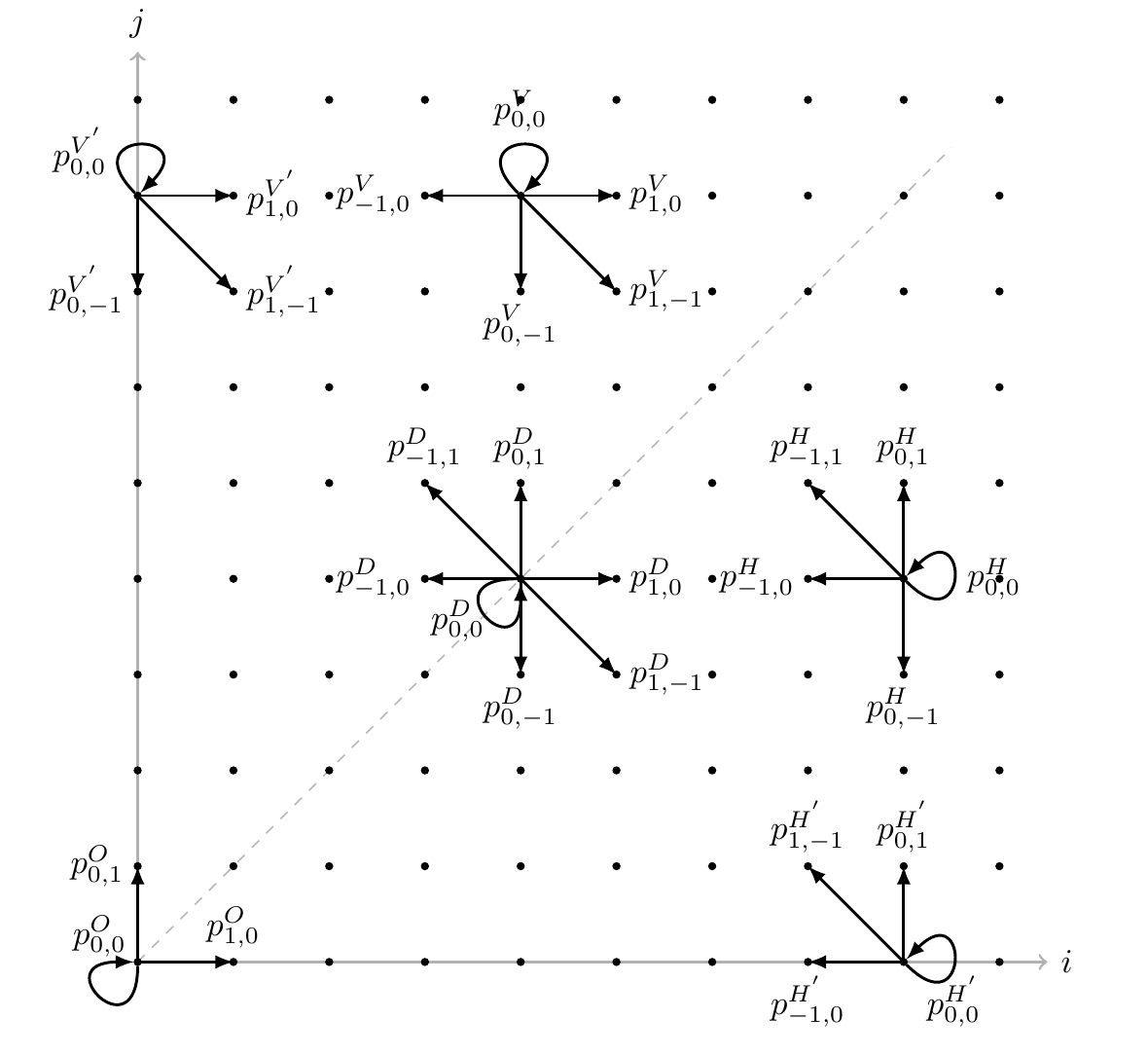}
		\caption{The transition probability diagram of $\{\bm{Q}(n),\, n \geq 0\}$ for a few representative states}
		\label{rw}
	\end{figure}
	
	From Figure \ref{rw}, it is evident that $\{\bm{Q}(n),\, n \geq 0\}$ has six regions of spatial homogeneity. Namely, two angles: upper-diagonal angle $H=\{(i,j):i>j>0\}$, and lower-diagonal angle $V=\{(i,j):j>i>0\}$; three rays: horizontal ray $H^{\prime}=\{(i,0):i>0\}$, vertical ray $V^{\prime}=\{(0,j):j>0\}$, and diagonal ray $D=\{(i,j):j=i>0\}$; and the origin point $O=(0,0)$. These six regions govern the one step transition probabilities, say $$p^{L}_{i',j'}=\mathbb{P}\left[\bm{Q}(n+1)=(i+i',j+j')\, |\, \bm{Q}(n)=(i,j)\in L\right], ~L\in\{H,V,H^{\prime},V^{\prime},D,O\}, ~i',j'=0,\pm 1,\pm2,$$
	\begin{enumerate}
		\item For $(i,j)\in \{H,V\}$,
		\begin{align}
		p^V_{1,0}&=p^H_{0,1}=\lambda(\bar{a}^{2}+a^{2}),\,p^V_{0,-1}=p^H_{0,-1}=\bar{\lambda}a\bar{a},\,p^V_{-1,0}=p^H_{-1,0}=\bar{\lambda}a\bar{a},\,p^V_{1,-1}=p^H_{-1,1}=\lambda a\bar{a}, \nonumber \\p^V_{0,0}&=p^H_{0,0}=\bar{\lambda}(\bar{a}^{2}+a^{2})+\lambda a\bar{a},\label{TProb1}
		\end{align}
		with $\bar{a}=1-a$ and $\bar{\lambda}=1-\lambda$.
		\item For $(i,j)\in \{H^{\prime},V^{\prime}\}$,
		\begin{align}
		p^{V^{'}}_{1,0}=p^{H^{'}}_{0,1}=\lambda(\bar{a}^{2}+a^{2}),\,p^{V^{'}}_{0,-1}=p^{H^{'}}_{-1,0}=\bar{\lambda}a,\,p^{V^{'}}_{1,-1}=p^{H^{'}}_{-1,1}=\lambda a\bar{a},\,p^{V^{'}}_{0,0}=p^{H^{'}}_{0,0}=\bar{\lambda}\bar{a}+\lambda a\bar{a}.\label{TProb2}
		\end{align}
		\item For $(i,j)\in D$,
		\begin{align}
		p^D_{1,0}&=p^D_{0,1}=\frac{1}{2}\lambda(\bar{a}^{2}+a^{2}),\,p^D_{0,-1}=p^D_{-1,0}=\bar{\lambda}\bar{a}a,p^D_{0,0}=\bar{\lambda}(\bar{a}^{2}+a^{2})+\lambda a\bar{a}, \nonumber \\
		p^D_{1,-1}&=p^D_{-1,1}=\frac{1}{2}\lambda\bar{a}a.\label{TProb3}
		\end{align}
		\item For $(i,j)\in O$,
		\begin{align}
		p^O_{0,1}&=p^O_{1,0}=\frac{1}{2}\lambda\bar{a},\,p^O_{0,0}=\bar{\lambda}+\lambda a.\label{TProb4}
		\end{align}
	\end{enumerate}

	\begin{remark}
		To better understand how the above probabilities are computed, we consider for example the case of $p_{1,0}^V$. This probability captures the transition from a state $(i,j) \in V$ to a state $(i+1,j)$. In this case, at the beginning of a slot, due to the EAS system, with probability $\lambda$ there is a new arrival. For the transition $(i,j) \to (i+1,j)$ to occur it is necessary on top of the arrival that no departure occurs. The latter event happens with probability $\bar{a}^2+a^2$. Thus, $p_{1,0}^V = \lambda(\bar{a}^2+a^2)$. The other transition probabilities are obtained in a similar fashion.
	\end{remark}
	\begin{remark}
		Note that if both relay queues are non empty, then the successful transmission rate from each of them equals $\bar{a}a$. If one of them is empty then the other transmits with rate $a$. This demonstrates that  the setting under consideration is a non-work-conservative setting. The only exception is the case   $2\bar{a}a=a  \iff a = 1/2 = \bar{a}$ (or equivalently the case $2\bar{a}a=\bar{a}$). Due to the non-work conservation setting, our model incorporates two features, that of the JSRQ and that of the ``coupled processors" \cite{fay}. The combination of the JSRQ feature and the coupled processor feature considerably complicates the analysis.
	\end{remark}

	\subsection{Stability condition}
	Let $(\mathbb{E}_{x}^{k},\mathbb{E}_{y}^{k})$ denote the mean jump vector of $\{\bm{Q}(n),\,n \geq 0\}$ in the angles $k=H,V$ or in the rays $k=H^{\prime},V^{\prime},D$. Then, it is readily derived that
	\begin{align*}
	\mathbb{E}_{x}^{H}&=\mathbb{E}_{y}^{V}=-(\lambda a\bar{a}+\bar{\lambda}a\bar{a})=-a\bar{a}<0,\\
	\mathbb{E}_{y}^{H}&=\mathbb{E}_{x}^{V}=\lambda(a^{2}+\bar{a}^{2})+\lambda a\bar{a}-\bar{\lambda}a\bar{a}=\lambda-a\bar{a},\\
	\mathbb{E}_{x}^{H^\prime}&=\mathbb{E}_{y}^{V^\prime}=-a(\bar{\lambda}+\lambda\bar{a})<0,\\
	\mathbb{E}_{y}^{H^\prime}&=\mathbb{E}_{x}^{V^\prime}=\lambda(\bar{a}^{2}+a^{2}+a\bar{a}),\\
	\mathbb{E}^{D}_{x}&=\mathbb{E}^{D}_{y}=\frac{\lambda}{2}(a^{2}+\bar{a}^{2})-\bar{\lambda}a\bar{a}=\frac{1}{2}(\lambda-2a\bar{a}).
	\end{align*} 
	Note that $(\mathbb{E}^{D}_{x},\mathbb{E}^{D}_{y})=\frac{1}{2}(\mathbb{E}_{x}^{H}+\mathbb{E}_{x}^{V},\mathbb{E}_{y}^{H}+\mathbb{E}_{y}^{V})$ and that $\mathbb{E}_{x}^{H}+\mathbb{E}_{y}^{H}=\mathbb{E}_{x}^{V}+\mathbb{E}_{y}^{V}=\lambda-2a\bar{a}$. Following the analysis in \cite{kurk}, we prove the following theorem.
	\begin{theorem}\label{thm1}
		The system at hand is stable if and only if
		\begin{equation}
		\mathbb{E}_{x}^{H}<0,\ 	\mathbb{E}_{y}^{V}<0,\
		\mathbb{E}_{x}^{H}+\mathbb{E}_{y}^{H}=\mathbb{E}_{x}^{V}+\mathbb{E}_{y}^{V}=\lambda-2a\bar{a}<0.
		\label{ui}
		\end{equation}
		Equivalently, the stability condition for the system can be written in terms of the system load
		\begin{equation}\label{Eq:Stab}
		\rho = \frac{\lambda(\bar{a}^{2}+a^{2})}{2\bar{\lambda}\bar{a}a}<1.
		\end{equation}
	\end{theorem} 
	\begin{proof} The proof consists of two parts: in the first part, we show that the stated condition is sufficient and in the second part, we show that the stated condition is necessary.
		\begin{description}
			\item{\bf{Sufficiency.}} To this purpose, we use Foster's criterion \cite{fayolle}:
			A Markov chain is ergodic provided that there exists a positive function $f(x,y)$ on $\mathbb{Z}_{+}^{2}$, a number $\epsilon>0$, and a finite set $A\in \mathbb{Z}_{+}^{2}$ such that
			\begin{equation}
			\mathbb{E}(f(x+\theta_{x},y+\theta_{y})-f(x,y))<-\epsilon,\,(x,y)\in \mathbb{Z}_{+}^{2}\smallsetminus A,
			\label{dc}
			\end{equation}
			with $(\theta_{x},\theta_{y})$ a random vector distributed as the one step jump of the Markov chain from state $(x,y)$.
			
			Assume that \eqref{ui} holds and consider the function $f(x,y)=\sqrt{x^{2}+y^{2}}$, which satisfies,
			\begin{equation*}
			\mathbb{E}(f(x+\theta_{x},y+\theta_{y})-f(x,y))=\frac{x\mathbb{E}(\theta_{x})+y\mathbb{E}(\theta_{y})}{f(x,y)}+o(1), \,\,as \,\, x^{2}+y^{2}\to\infty.
			\end{equation*}
			If $x>y>0$, $\mathbb{E}(\theta_{x})=\mathbb{E}_{x}^{H}<0$, and $\mathbb{E}(\theta_{x})+\mathbb{E}(\theta_{y})=\mathbb{E}_{x}^{H}+\mathbb{E}_{y}^{H}=\lambda-2a\bar{a}<0$, and 
			\begin{equation}
			\mathbb{E}(f(x+\theta_{x},y+\theta_{y})-f(x,y))\leq\frac{y(\mathbb{E}_{x}^{H}+\mathbb{E}_{y}^{H})}{y\sqrt{2}}+o(1)<-\epsilon_{0},
			\label{th1}
			\end{equation}
			for some $\epsilon_{0}>0$, and for all $(x,y):x^{2}+y^{2}\to\infty.$
			
			Assume now that $x>0$, $y=0$. Then since $\mathbb{E}_{x}^{H^\prime}<0$,
			\begin{equation*}
			\mathbb{E}(f(x+\theta_{x},y+\theta_{y})-f(x,y))\leq\frac{x\mathbb{E}_{x}^{H^\prime}}{x}+o(1)<-\epsilon_{1},
			\end{equation*}
			for some $\epsilon_{1}>0$, and for all $x$ sufficiently large.
			
			The case $y>x$ is symmetric to $x>y$ and further details are omitted. For $x=y>0$, since $(\mathbb{E}^{D}_{x},\mathbb{E}^{D}_{y})=\frac{1}{2}(\mathbb{E}_{x}^{H}+\mathbb{E}_{x}^{V},\mathbb{E}_{y}^{H}+\mathbb{E}_{y}^{V})$, we can check \eqref{dc} using again \eqref{th1}. Then, Foster's criterion applies and the chain is ergodic. 
			\item{\bf{Necessity.}} It is sufficient to show that there exists a function $f(x,y)$ on $\mathbb{Z}_{+}^{2}$ and a constant $c>0$ such that
			\begin{equation}
			\mathbb{E}(f(x+\theta_{x},y+\theta_{y})-f(x,y))\geq0,\,(x,y):\,f(x,y)>c.
			\label{dc1}
			\end{equation}
			
			Let $\mathbb{E}_{x}^{H}+\mathbb{E}_{y}^{H}=\mathbb{E}_{x}^{V}+\mathbb{E}_{y}^{V}=\lambda-2a\bar{a}\geq0$. Set $f(x,y)=x+y$. If $x,y>0$,
			\begin{equation*}
			\mathbb{E}(f(x+\theta_{x},y+\theta_{y})-f(x,y))=\mathbb{E}_{x}^{H}+\mathbb{E}_{y}^{H}\geq0.
			\end{equation*}
			If $x>0$, $y=0$,
			\begin{equation*}
			\mathbb{E}(f(x+\theta_{x},y+\theta_{y})-f(x,y))=\mathbb{E}_{x}^{H^\prime}+\mathbb{E}_{y}^{H^\prime}=\lambda(\bar{a}^{2}+a^{2})-\bar{\lambda}a>0.
			\end{equation*}
			Similarly, if $x=0$, $y>0$,
			\begin{equation*}
			\mathbb{E}(f(x+\theta_{x},y+\theta_{y})-f(x,y))=\mathbb{E}_{x}^{V^\prime}+\mathbb{E}_{y}^{V^\prime}=\lambda(\bar{a}^{2}+a^{2})-\bar{\lambda}a>0.
			\end{equation*}
			Thus, $f(x,y)$ satisfies \eqref{dc1} and the chain is non-ergodic.
		\end{description}\vspace{-1cm}
	\end{proof}
	
\begin{remark}
		Theorem \ref{thm1} can be generalised  to the asymmetric case following \cite{kurk}. This analysis would reveal that the stability condition is $\lambda-a_1(1-a_2)-(1-a_1)a_2<0$, or equivalently in terms of the system load
		$$ \f{\l(a_1a_2 + (1-a_1)(1-a_2)}{(1-\l)(a_1(1-a_2) + a_2(1-a_1))} < 1.$$
	\end{remark}
	
	\section{Equilibrium analysis: compensation approach}
	\label{Sec4}
	
	The compensation approach is developed by Adan et al. in a series of papers \cite{ad0,ad1,ad5} and aims at a direct solution for the equilibrium joint queue length distribution, for a sub-class of two-dimensional random walks on the lattice of the first quadrant obeying the following conditions:
	\begin{itemize}
		\item Homogeneity: The same transitions occur according to the same rates for all interior points, and similarly for all points on the horizontal boundary, and for all points on the vertical boundary.
		\item Forbidden steps: No transitions from interior states to the North, North-East, and East.
		\item  Step size: Only transitions to neighbouring states.
	\end{itemize}
	It exploits the fact that the balance equations in the interior of the quarter plane are satisfied by a linear combination (finite or infinite) of product-form terms, the parameters of which satisfy a kernel equation, and that need to be chosen such that the balance equations on the boundaries are satisfied as well. As it turns out, this can be done by alternatingly compensating for the errors on the two boundaries, which eventually leads to an infinite series of product-forms.

	As evident from Figure \ref{rw}, the model at hand violates the first two conditions mentioned above. In order to apply the compensation approach, it is necessary to transform the state space (similarly as in the case of  the classical JSQ model \cite{ad0,ad1}). More concretely, we employ the following transformation 
	\begin{displaymath}
	\tilde{Q}_{1}(n)= \min\left\{Q_{1}(n),Q_{2}(n)\right\},\,\tilde{Q}_{2}(n)= \left|Q_{2}(n)-Q_{1}(n)\right|.
	\end{displaymath} 
	Clearly, $\{\tilde{\bm{Q}}(n),\,n \geq 0\}:=\{(\tilde{Q}_{1}(n),\tilde{Q}_{2}(n)),\,n \geq 0\}$ is a discrete time Markov chain with state space $\tilde{\mathcal{S}}=\{(k,l):\,k,l\geq0\}$. The corresponding probability transition diagram is depicted in Figure \ref{Trw}.
	
	\begin{figure}[h!]
		\centering	\includegraphics[scale= 0.9]{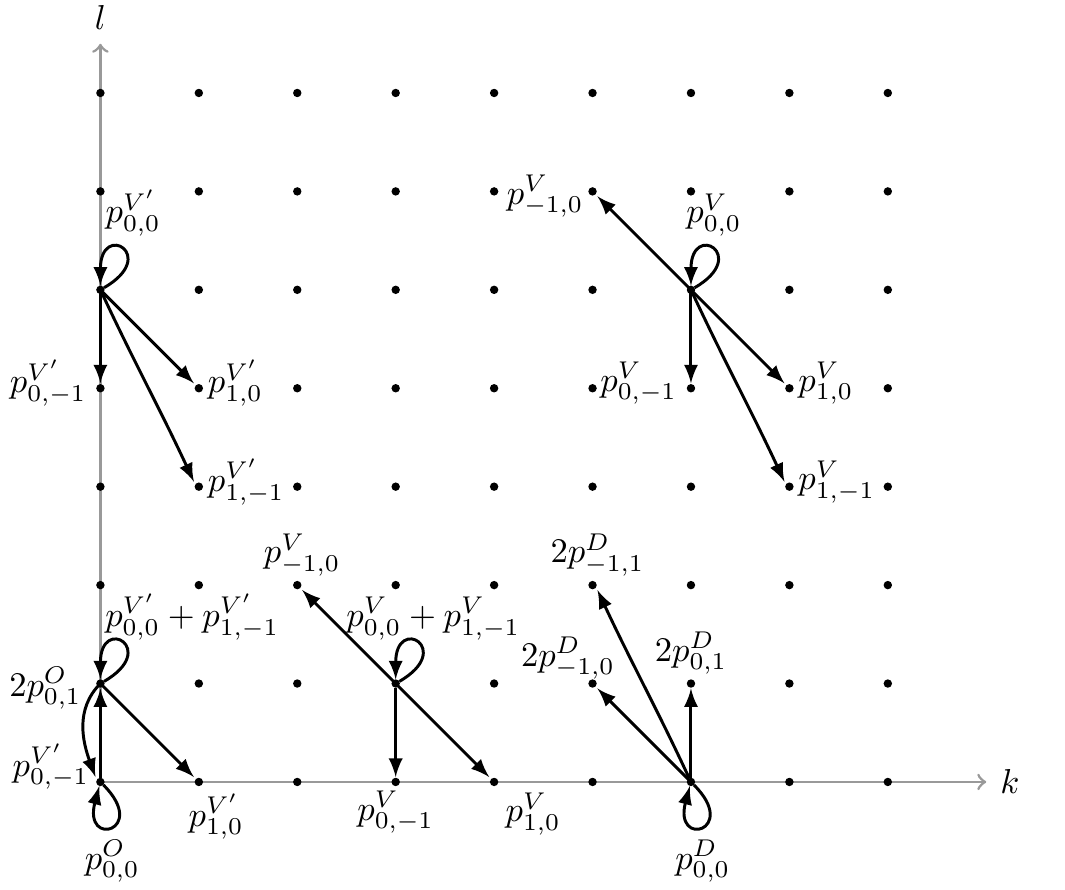}
		\caption{The probability transition  diagram of $\{\tilde{\bm{Q}}(n),\, n \geq 0\}$ for a few representative states}
		\label{Trw}
	\end{figure}

	Note that the above transformation has led to a random walk that only violates the nearest neighbour condition for the application of the compensation approach, but the new random walk has bounded transitions, see Figure \ref{Trw}. Despite violating the nearest neighbour condition, we show, in this paper, that the compensation approach can still be applied and lead to an equilibrium joint queue length distribution expressed in the form of a linear combination (infinite) of product-form terms.
	
	Let 
	\begin{displaymath}
	\pi_{k,l}=\lim_{n\to\infty}\mathbb{P}(\tilde{Q}_{1}(n)=k,\tilde{Q}_{2}(n)=l)
	=\lim_{n\to\infty}\mathbb{P}(\min\{{Q}_1(n),{Q}_2(n)\}=k, |{Q}_1(n)-{Q}_2(n)|=l), \ k,l\geq0,
	\end{displaymath}
	denote the equilibrium joint queue length distribution of the transformed random walk $\{\tilde{\bm{Q}}(n),\,n \geq 0\}$.
	
	The balance equations read as follows
	\begin{align}
	\pi_{k,l}=&\pi_{k,l}(\bar{\lambda}(\bar{a}^{2}+a^{2})+\lambda a\bar{a})+\pi_{k,l+1}\bar{\lambda}a\bar{a}+\pi_{k-1,l+1}\lambda(\bar{a}^{2}+a^{2})\nonumber\\
	&+\pi_{k-1,l+2}\lambda\bar{a}a+\pi_{k+1,l-1}\bar{\lambda}\bar{a}a,\,\,\,k\geq1,\,l\geq3,\label{in}\\
	\pi_{k,0}=&\pi_{k,0}(\bar{\lambda}(\bar{a}^{2}+a^{2})+\lambda a\bar{a})+\pi_{k,1}\bar{\lambda}a\bar{a}+\pi_{k-1,1}\lambda(\bar{a}^{2}+a^{2})+\pi_{k-1,2}\lambda a\bar{a},\,\,\,k\geq1,\label{hor1}\\
	\pi_{k,1}=&\pi_{k,1}(\bar{\lambda}(\bar{a}^{2}+a^{2})+2\lambda a\bar{a})+\pi_{k,2}\bar{\lambda}a\bar{a}+\pi_{k-1,2}\lambda(\bar{a}^{2}+a^{2})+\pi_{k-1,3}\lambda a\bar{a} \nonumber \\ 
	&+\pi_{k,0}\lambda(\bar{a}^{2}+a^{2})+\pi_{k+1,0}2\bar{\lambda}a\bar{a},\,\,\,k\geq1,\label{hor2}\\
	\pi_{k,2}=&\pi_{k,2}(\bar{\lambda}(\bar{a}^{2}+a^{2})+\lambda a\bar{a})+\pi_{k,3}\bar{\lambda} a\bar{a}+\pi_{k-1,3}\lambda(\bar{a}^{2}+a^{2}) \nonumber \\
	&+\pi_{k-1,4}\lambda a\bar{a}+\pi_{k+1,0}\lambda a\bar{a}+\pi_{k+1,1}\bar{\lambda}a\bar{a},\,\,\,k\geq1,
	\label{hor3}\\
	\pi_{0,l}=&\pi_{0,l}(\bar{\lambda}\bar{a}+\lambda a\bar{a})+\pi_{0,l+1}\bar{\lambda}a+\pi_{1,l-1}\bar{\lambda}a\bar{a},\,\,\,\,l\geq 3,\label{ver}\\
	\pi_{0,2}=&\pi_{0,2}(\bar{\lambda}\bar{a}+\lambda a\bar{a})+\pi_{0,3}\bar{\lambda}a+\pi_{1,1}\bar{\lambda}a\bar{a}+\pi_{1,0}\lambda a\bar{a},\label{RB1}\\
	\pi_{0,1}=&\pi_{0,1}(\bar{\lambda}\bar{a}+2\lambda a\bar{a})+\pi_{0,2}\bar{\lambda}a+\pi_{1,0}2\bar{\lambda}a\bar{a}+\pi_{0,0}\lambda\bar{a},\label{RB2}\\
	\pi_{0,0}=&\pi_{0,0}(\bar{\lambda}+\lambda a)+\pi_{0,1}\bar{\lambda}a.\label{RB3}
	\end{align}

	\subsection{Decay rate\label{section_Decay_Rate}}
	From the exact analysis of the system, see Section \ref{CA} and the result therein, it turns out that, although the model at hand combines both the JSRQ feature and the coupled processor feature, the dominating feature is that of the JSRQ. This is already noticeable in the decay result of the following proposition, indicating a behaviour similar to that of the classical JSQ, cf. \cite{ad0}.
	\begin{proposition}\label{propDecay} 
		For $\rho<1$,
		\begin{equation} 
		\lim_{k\rightarrow\infty}\rho^{-2k}\mathbb{P}(\min\{{Q}_1,{Q}_2\}=k, |{Q}_1-{Q}_2|=l)=c\,d_0\, \delta^l,
		\end{equation}
		with $c,d_0, \delta$ constants that do not depend on $k$.
	\end{proposition}
	Proposition \ref{propDecay} can be intuitively understood by comparing the transformed model with the corresponding $Geo/Geo/2$ model with one queue, Bernoulli arrivals with success probability $\lambda$, two identical servers, each with Bernoulli service with success probability $\bar{a}a$. Let ${Q}_1$ and ${Q}_2$ denote the equilibrium queue lengths in the original random walk, and let ${Q}$ denote the equilibrium queue length in the $Geo/Geo/2$ model with one queue. It is then expected that $\mathbb{P}({Q}_1+{Q}_2=k)$ and $\mathbb{P}({Q}=k)$ have the same decay rate, since both systems will work at full capacity whenever the total number of customers grows large. Moreover, since the JSRQ protocol constantly aims at balancing the lengths of the two queues over time, it is expected that, for large values of $k$, 
	\begin{equation}
	\label{h1} 
	\mathbb{P}(\min\{{Q}_1,{Q}_2\}=k)\approx\mathbb{P}({Q}_1+{Q}_2=2k)\approx\mathbb{P}({Q}=2k).
	\end{equation}
	For the standard $Geo/Geo/2$ model with one queue, it is well known that 
	\begin{equation}\label{h2} \mathbb{P}({Q}=k)\approx(1-\rho)\rho^k.
	\end{equation}
	Combining \eqref{h1} and \eqref{h2} leads to the following conjectured behaviour of the tail probability of the minimum queue length
	\begin{equation}\label{h3} \mathbb{P}(\min\{{Q}_1,{Q}_2\}=m)\approx  C \rho^{2k},\ k\to\infty, \end{equation}
	for some positive constant $C$. Hence, the decay rate of the tail probabilities for $\min\{Q_1,Q_2\}$ is conjectured to be equal to the square of the decay rate of the tail probabilities of $Q$. Proposition \ref{propDecay} states this conjecture, for the case when the difference of the queue sizes is fixed.
	
	Furthermore, one can determine the decay rate of the marginal distribution of $\min(Q_1,Q_2)$. The latter does not follow immediately from Proposition \ref{propDecay}, because the summation over the difference in queue sizes, which can be unbounded, requires a formal justification. In this paper, we derive an exact expression for $\pi_{k,l}$, which, among other things renders the following result.
	\begin{proposition}\label{prop-3.2}
		For $\rho<1$,
		\begin{equation} 
		\lim_{k\rightarrow\infty}\rho^{-2k}\mathbb{P}(\min(Q_1,Q_2)=k)=\frac{c\,d_0}{1- \delta},
		\end{equation}
		with $c,d_0, \delta$ constants that do not depend on $k$.
	\end{proposition}
	The proofs of Propositions \ref{propDecay}  and \ref{prop-3.2}, along with several other asymptotic and exact results, are given in Section \ref{CA}.

	\subsection{The compensation procedure}\label{CA}
	In this section, we obtain the equilibrium joint queue length distribution using the compensation approach. This approach  yields an explicit expression by directly exploiting the balance equations, without any transforms. 
	
	The compensation approach as has been briefly discussed in the introduction attempts to solve the balance equations by a linear combination of product-form terms. This is achieved by first characterising a sufficiently rich basis of product-form solutions satisfying the balance equations in the interior of the state space. Subsequently this basis is used to construct a linear combination that also satisfies the equations for the boundary states. Note that the basis contains uncountably many elements. Therefore a procedure is needed to select the appropriate elements. This procedure is based on a compensation argument (which explains the name of the method): after introducing the first term, countably many terms may subsequently be added so as to alternatingly compensate for the error on one of the two boundaries. The main steps of the analysis are briefly outlined below.

	\begin{description}
		\item[Step 1] Characterise the set of product-forms
		\[\gamma^{k}\delta^{l}\]
		satisfying the balance equations in the interior of the state space, i.e., Equation \eqref{in}.
		Substitution of the product-form $\gamma^{k}\delta^{l}$ into \eqref{in} and division by common powers yields a cubic equation in $\gamma$ and $\delta$
		\begin{equation}
		\begin{array}{rl}
		\left(1-\left(\bar{\lambda}(\bar{a}^2+a^2) + \lambda a\bar{a}\right) \right)\gamma\delta=&  \left(\bar{\lambda} \bar{a} a \gamma  + \lambda(a^2 + \bar{a}^2)\right) \delta^2 + \lambda \bar{a} a  \delta^3 + \bar{\lambda} \bar{a} a \gamma^2.
		\end{array}\label{eq}
		\end{equation}
		The solutions to Equation \eqref{eq} form the basis. In particular the points on the curve \eqref{eq} and  inside
		the region $0<|\gamma|,|\delta|<1$ characterise a continuum of product-forms satisfying the inner equations.
		
		\item[Step 2] Construct a linear combination of elements in this rich basis, which is a formal solution to the balance equations. Here the word formal is used to indicate that (at this stage) we do not treat the convergence of the solution. This aspect is treated in Step 3. The formal solution is constructed as follows:
		\begin{enumerate}
			\item[(a)]  The construction of a linear combination starts with a 
			suitable initial term $\gamma_0$ that satisfies the interior of the state space and also the balance equations   \eqref{hor1}-\eqref{hor3}. In Lemma \ref{lem1s}, in Appendix \ref{ApendixCA}, we prove that $\gamma_0=\rho^2$. Then, from Equation \eqref{eq} we obtain the unique $\delta_0$, with $|\delta_0|<|\gamma_0|$, such that
			\[ \pi_{k,l} \vDash d_0 \gamma_0^k\delta_0^l,\, k>0,l\geq0,\]
			satisfies Equations \eqref{in}-\eqref{hor3}, where the  double turnstile symbol $( \vDash )$ is used to signify that $ \pi_{k,l}$ semantically entails the form $d_0\gamma_0^k\delta_0^l$.  The uniqueness of the $\delta$-root is proven in Lemma \ref{lemma1}. The constant $d_0$ can be set equal to one, this will be corrected in Step 4 with the computation of the normalisation constant.
			\item[(b)]  The starting tuple $(\gamma_0,\delta_0)$ violates Equation \eqref{ver} on the vertical boundary. To compensate for this error, we add a new product-form term coming from the basis, such that  the sum of the two terms satisfies the balance equations in all states on the vertical boundary. In particular, it is easy to show that the new tuple is $(\gamma_1,\delta_0)$. The new $\gamma$-root is uniquely determined from Equation \eqref{eq}, with $|\delta_1|<|\gamma_0|$. The uniqueness of the $\gamma$-root is proven in Lemma \ref{lemma1}.
			
			Then,
			\[ \pi_{k,l} \vDash d_0\gamma_{0}^{k}\delta_{0}^{l}+c_1\gamma_{1}^{k}\delta_{0}^{l},\, k\geq0,l\geq3,\] 
			satisfies \eqref{ver} if
			\begin{align}\label{JSRQ_C}
			c_1=-\frac{\delta_{0}(1-\bar{\lambda}\bar{a}-\lambda a\bar{a})-\bar{\lambda}a\delta_{0}^{2}-\bar{\lambda}a\bar{a}\gamma_{0}}{\delta_{0}(1-\bar{\lambda}\bar{a}-\lambda a\bar{a})-\bar{\lambda}a\delta_{0}^{2}-\bar{\lambda}a\bar{a}\gamma_{1}}d_0,
			\end{align}
			with $d_0$ known constant from the previous step.
			
			\item[(c)] Adding the new term violates the balance equations \eqref{hor1}-\eqref{hor3}, hence we compensate for this error by adding a product-form solution $(\gamma_1,\delta_1)$ satisfying \eqref{eq}, and \eqref{hor1}-\eqref{hor3}, such that 
			\[ \pi_{k,l} \vDash 
			\begin{cases}
			d_0\gamma_{0}^{k}\delta_{0}^{l}+c_1\gamma_{1}^{k}\delta_{0}^{l}+d_1\gamma_{1}^{k}\delta_{1}^{l},\, k>0,l\geq2,\\
			d_0\gamma_{0}^{k}\delta_{0}^{l}+e_{l,1}\gamma_{1}^{k}\delta_{1}^{l},\, k>0,l=0,1.\\
			\end{cases}
			\]
			The three unknowns $e_{0,1}, e_{1,1}$ and $d_1$ can be computed from the following system of three equations 
			\begin{equation} \label{System1}
			A(\gamma_1,\delta_1) \begin{bmatrix} e_{0,1} \\ e_{1,1} \end{bmatrix}  + B(\gamma_1, \delta_1) d_1  \delta_1^2 = - B(\gamma_1, \delta_0) c_1 \delta_0^2, 
			\end{equation}
			with
			\begin{eqnarray*} 
				A(\gamma,\delta) &=& \begin{bmatrix} 
					\gamma(1-\bar{\lambda}(\bar{a}^{2}+a^{2})-\lambda a\bar{a}) & -(\gamma\delta\bar{\lambda}a\bar{a}+\delta \lambda(\bar{a}^{2}+a^{2}))\\
					-\gamma(\lambda(\bar{a}^{2}+a^{2})+2\gamma\bar{\lambda}a\bar{a}) & \gamma\delta(1-{\lambda}(\bar{a}^{2}+a^{2})-2\lambda a\bar{a}) \\
					\gamma^2\lambda a \bar{a}&\gamma^2\delta\bar{\lambda} a \bar{a}
				\end{bmatrix}, \\
				B(\gamma, \delta) &=& 
				\begin{bmatrix} 
					- \lambda a\bar{a} \\ 
					-(\gamma\bar{\lambda}a\bar{a}+\lambda(\bar{a}^{2}+a^{2})+\delta\lambda a\bar{a}) \\  
					-\gamma(1-\bar{\lambda}(\bar{a}^{2}+a^{2})-\lambda a\bar{a})+\gamma\delta\bar{\lambda a \bar{a}} +\delta{\lambda}(\bar{a}^{2}+a^{2})+\delta^2\lambda a \bar{a}
				\end{bmatrix}.
			\end{eqnarray*}
			
			\item[(d)] We  continue in this manner until we  construct the entire formal series
			\begin{align}
			\pi_{k,l} &\vDash \sum_{i=0}^{\infty}(d_{i}\gamma_{i}^{k}+c_{i+1}\gamma_{i+1}^{k})\delta_{i}^{l},\,~~ k\geq0,l\geq2,\,\text{ (pairs with the same $\d$-term)},\,\label{on}\\
			\pi_{k,l} &\vDash d_0\gamma_0^k\delta_0^l+\sum_{i=1}^{\infty}e_{l,i}\gamma_{i}^{k}\delta_0^l,\, ~~k>0,l=0,1,\label{on2}
			\end{align} 
			and $\pi_{0,0}$, $\pi_{0,1}$, and $\pi_{0,2}$ are obtained from Equations \eqref{RB1}-\eqref{RB3}. Note that Equation \eqref{on} can be equivalently written as follows
\begin{align*}
			\pi_{k,l} &\vDash d_0\gamma_0^k\delta_0^l+\sum_{i=0}^{\infty}(c_{i+1}\delta_{i}^{l}+d_{i+1}\delta_{i+1}^{l})\gamma_{i+1}^{k},\,~~k\geq0,l\geq2,\,\text{ (pairs with the same $\gamma$-term)}.
\end{align*} 			
			
		\end{enumerate}

		\item[Step 3] Prove that the formal solution \eqref{on} and \eqref{on2} converges. This is split up into two parts: i) we first show in Proposition \ref{JSQR_prop4} that the sequences $\{\gamma_i\}_{i\in\mathbb{N}}$ and $\{\delta_i\}_{i\in\mathbb{N}}$ converge to zero exponentially fast, and ii) we show in Proposition \ref{prop_Abs_Conv} that the formal solution converges absolutely in all states. The Propositions and their proofs are in Appendix \ref{ApendixCA}.
		
		\item[Step 4] Determine the normalisation constant.
	\end{description}
	Performing the steps described above for the compensation approach  leads to the following main result for the equilibrium distribution.
	
	\begin{theorem}\label{main}
		For $\rho<1$, 
		\begin{align}
		\pi_{k,l} &= c\sum_{i=0}^{\infty}(d_{i}\gamma_{i}^{k}+c_{i+1}\gamma_{i+1}^{k})\delta_{i}^{l},\,~~ k\geq0,l\geq3,\, \text{ (pairs with the same $\delta$-term)}, \label{on-final}\\
		&= c\Big(d_0\gamma_0^k\delta_0^l+\sum_{i=0}^{\infty}(c_{i+1}\delta_{i}^{l}+d_{i+1}\delta_{i+1}^{l})\gamma_{i+1}^{k}\Big),\, ~~k\geq0,l\geq3,\,\text{ (pairs with the same $\gamma$-term)},\nonumber\\
		\pi_{k,l} &= c \Big(d_0\gamma_0^k+\sum_{i=1}^{\infty}e_{l,i}\gamma_{i}^{k}\Big),\,~~ k>0,l=0,1,\label{on2-final}
		\end{align}
		with $c$ denoting the normalisation constant. 
		The sequences $\{\gamma_i\}_{i\in\mathbb{N}}$, $\{\delta_i\}_{i\in\mathbb{N}}$, $\{c_i\}_{i\in\mathbb{N}}$, $\{d_i\}_{i\in\mathbb{N}}$, and $\{e_i\}_{i\in\mathbb{N}}$ are obtained recursively based on the analysis of Steps 1-4 above. 
		
		The equilibrium probabilities close to the origin $\pi_{0,0}$, $\pi_{0,1}$, and $\pi_{0,2}$ are obtained by directly solving the balance equations \eqref{RB1}-\eqref{RB3}.
	\end{theorem}
	Clearly, from this result we can derive similar expressions for other performance characteristics such as the mean queue lengths, the correlation between the queue lengths, the mean waiting time, etc. 
	
	\begin{remark}
		Theorem \ref{main} can be generalised  to the asymmetric case by directly replicating Steps 1-4 above, see \cite{ad0,ad1}. 
	\end{remark}

	\subsection{Numerical implementation of the compensation approach} 
	In this section, we discuss how to numerically implement  the equilibrium distribution analysis based on the compensation approach. The equilibrium distribution $\pi_{k, l}$, $k\geq0$, $l\geq0$, is written as a linear combination of product-form terms, cf.  \eqref{on-final}, \eqref{on2-final}. The first step of the numerical implementation is to consider the first few terms of the series expression for $\pi_{k, l}$. More concretely, let
	\begin{align}
	\pi_{k,l}^{(N_{\mathrm{ca}})} &= c\sum_{i=0}^{N_{\mathrm{ca}}}(d_{i}\gamma_{i}^{k}+c_{i+1}\gamma_{i+1}^{k})\delta_{i}^{l},\, ~~ k\geq0,l\geq3,\label{on-numer}\\
	\pi_{k,l}^{(N_{\mathrm{ca}})} 	&= c\Big(d_0\gamma_0^k+\sum_{i=1}^{N_{\mathrm{ca}}}e_{l,i}\gamma_{i}^{k}\Big),\,~~ k>0,l=0,1.\label{on2-numer}
	\end{align}
	where here $N_{\mathrm{ca}}$ denotes the truncation level of the series expression obtained using the compensation approach. From the analysis of Proposition \ref{prop_Abs_Conv}, it follows that the inclusion of more terms of the series expression improves to a desired accuracy the computation of the equilibrium distribution, i.e., the bigger the value of $N_{\mathrm{ca}}$ the better the approximation of $\pi_{k, l}$. The constant $N_{\mathrm{ca}}$ is determined such that 
	\begin{equation}
	\label{Eq:stop_crit}
	\left| \frac{\sum\limits_{k,l = 0}^{N_{\mathrm{ca}}} \pi^{{(N_{\mathrm{ca}})}}_{k, l} - \sum\limits_{k,l = 0}^{N_{\mathrm{ca}}} \pi^{({N_{\mathrm{ca}}}-1)}_{k, l}}{\sum\limits_{k,l = 0}^{N_{\mathrm{ca}}} \pi^{{(N_{\mathrm{ca}})}}_{k,l}}\right| < \e,\end{equation}
	with $\e$ the precision error. This is included as the stopping criterium for the algorithm, i.e., we start with  $N_{\mathrm{ca}}=1$ and as long as Equation \eqref{Eq:stop_crit} is not satisfied, we increase the value of $N_{\mathrm{ca}}$ by one.
	
	Furthermore, for numerical purposes, we assume that the state-space  $\tilde{\mathcal{S}}$ is truncated, i.e., we consider the following truncated state-space $\tilde{\mathcal{S}}^{(T^{(k)},T^{(l)})}=\{(k,l):\,0\leq k\leq T^{(k)},0\leq l\leq T^{(l)}\}$. The constants $T^{(k)}$, $T^{(l)}$ are determined such that 
	\begin{align*}
	c \,d_0\, \gamma_0^{T^{(k)}}<\e \text{ and } c \,d_0\, \delta_0^{T^{(l)}}<\e,
	\end{align*}
	and $T^{(k)},T^{(l)}\geq3$, so as Equations  \eqref{on-final} and \eqref{on2-final} can be applied. 
	Note that the above are direct consequences of the asymptotic behaviour of the random walk at hand, cf. Proposition \ref{propDecay}. Furthermore, as $|\delta_0|<|\gamma_0|$, cf. Proposition \ref{JSQR_prop4}, it suffices to choose 
	\begin{align*}
	T^{ca}:=T^{(k)}=T^{(l)} \approx \max\{\lceil \log(\e)/\log(\d_{0}) \rceil,3\}.
	\end{align*}
	
	In Algorithm \ref{Compensation_Algo}, we provide all the necessary information for the numerical implementation of the compensation approach.

	\begin{algorithm}[H]	
		\caption{Compensation approach algorithmic implementation}
		\label{Compensation_Algo}
		\begin{algorithmic}	[1]	
			\State Inputs $\l$, $a$ and precision $\e$.
			\State \label{Algo_CA_Step2}Set $\gamma_0 = \rho^2 $, $d_0=1$ and $N_{\mathrm{ca}} = 1$.
			\State   Compute  $\d_0$ from Equation \eqref{eq}.
			\State \label{Algo_CA_Step4}Set $T_{\mathrm{ca}}  =\max\{ \lceil \log(\e)/\log(\d_{0}) \rceil,3\}$. 
			\State \label{Algo_CA_Step5}Compute recursively  $\g_i$, $\d_i$, for $i=1,\ldots, N_{\mathrm{ca}}$, from Equation \eqref{eq}.
			\State  \label{Algo_CA_Step6}Compute the coefficients $c_i, e_{i, 0}, e_{i, 1}$ and $d_i$, $i=0,1,\ldots, N_{\mathrm{ca}}$, recursively from Equations \eqref{JSRQ_C} and \eqref{System1}, starting with $d_0=1$, cf. Step \ref{Algo_CA_Step2}.
			\State \label{Algo_CA_Step6}For all $ \lfloor T_{\mathrm{ca}} / 2\rfloor < k, l \leq T_{\mathrm{ca}}$, compute  $\pi^{(N_{\mathrm{ca}})}_{k, l}$  from Equation \eqref{on2-numer}.
			\State \label{Algo_CA_Step7}\label{Algo_CA_Step8}For all $ 0 \leq k, l \leq \lfloor T_{\mathrm{ca}} / 2\rfloor$, solve the linear system of the balance equations \eqref{in}-\eqref{RB3} and compute  $\pi^{N_{\mathrm{ca}}}_{k, l}$.
			\State   \label{Algo_CA_Step8}Normalize  $\pi^{(N_{\mathrm{ca}})}_{k, l}$, i.e., set   $\pi_{k, l}^{(N_{\mathrm{ca}})} = \pi_{k, l}^{(N_{\mathrm{ca}})}/\sum\limits_{0\leq k, l\leq T_{\mathrm{ca}}} \pi_{k, l}^{(N_{\mathrm{ca}})}$, for all $0 \leq k, l \leq {T_{\mathrm{ca}}}$.
			\State  Stop if ${\left| \frac{\sum\limits_{k,l = 0}^{T_{\mathrm{ca}}} \pi^{{(N_{\mathrm{ca}})}}_{k, l} - \sum\limits_{k,l = 0}^{T_{\mathrm{ca}}} \pi^{({N_{\mathrm{ca}}}-1)}_{k, l}}{\sum\limits_{k,l = 0}^{T_{\mathrm{ca}}} \pi^{{(N_{\mathrm{ca}})}}_{k,l}}\right|} < \e$, else update  $N_{\mathrm{ca}} = N_{\mathrm{ca}}+1$ and go to Step \ref{Algo_CA_Step5}.		
		\end{algorithmic}		
	\end{algorithm}

	\subsection{Applicability of the compensation approach in case of bounded transitions}\label{com}
	In Section \ref{Sec4}, we mentioned that the model at hand violates the nearest neighbour condition for the applicability of the compensation approach. Nonetheless, the analysis performed demonstrated that the compensation approach can be generalised to cover a larger class of random walks permitting transitions not only to the nearest neighbours, but to a bounded region of neighbouring states. From our analysis, it becomes clear that for random walks with the structure of the system at hand and bounded quasi-birth-death transitions along the rays $R_L=\{(k,l):\, 2k+l=L\}$, $L\geq L_0$, the analysis performed in Section \ref{CA} carries out to a tee. We have confirmed this in the system depicted in Figure \ref{3hop_fig}.
	
	\begin{figure}[h!]
		\begin{center}
			\includegraphics[scale= 0.9]{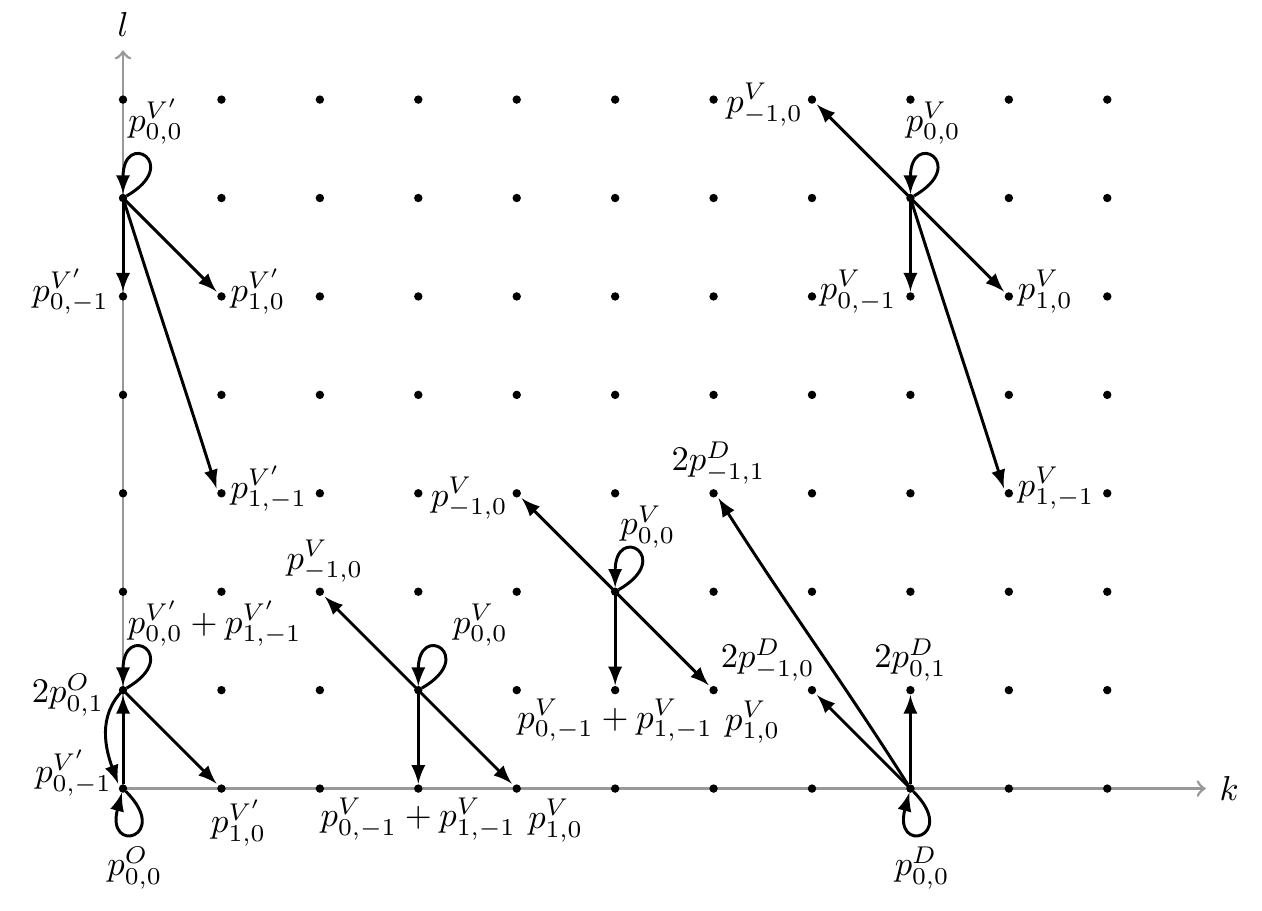}
		\end{center}	
		\caption{The one step transition probabilities diagram for a few representative states}
		\label{3hop_fig}
	\end{figure}

	\section{Equilibrium analysis: Power series algorithm implementation}
	\label{Sec:PSA}
	In this section, we show how the power series algorithm (PSA) can be used to analyse the relay-based cooperative communication network with collisions and with JSRQ protocol. PSA is an algorithmic procedure which is often used to numerically obtain the performance measures of multi-dimensional queueing models, which fit in the class of quasi birth-and-death processes. The intrinsic idea behind PSA is the transformation of the non-recursively solvable set of balance equations into a recursively solvable set of equations by adding one dimension into the state space. This is achieved by expressing the equilibrium distribution as power series in some variable based on the model parameters, this allows the calculation of the equilibrium joint queue length distribution. PSA was first applied by Beneš \cite{benevs1965} and thereafter by Hooghiemstra et al. \cite{hooghiemstra1988power}, and it was further developed by Blanc and co-authors, see, e.g., \cite{blanc1987numerical,blanc1992,Koole1994}. Although this procedure lacks in theoretical foundation, it has been very successfully applied to several systems with multiple queues. One of the objectives of this work is to provide an extensive numerical comparison between the compensation approach and PSA. For this reason, we show how PSA is applied to the two relay model at hand.
	
	\begin{remark}
		Note that PSA is a powerful numerically oriented procedure, that can be applied to the asymmetric case and it can be generalised to any (finite) number of relays. 
	\end{remark} 
	
	\subsection{Computation scheme} 
	For the analysis that follows, we use the transformed model, see Section \ref{Sec4}, and the balance equations \eqref{in}-\eqref{RB3}. For the application of the PSA procedure, we observe the following property.
	
	\begin{property}
		For each state $(k,l)\in\hat{S}$, it holds that 
		\begin{equation}\label{PGF}
		\pi_{k, l} =  \sum\limits_{n = 0}^{\infty} ~\rho^{n + k + l}~ \beta(n, k, l),\ k,l\geq0,
		\end{equation} 
		under the following assumptions:
		\begin{itemize}
			\item [A1.] $0\leq \rho <1$ (stability condition);
			\item [A2.] $\pi_{k, l}$ can be analytically continued as a function of $\rho$ into a domain which includes the disk $|\rho -\frac{1}{2} | \leq \frac{1}{2} $.
		\end{itemize}
		As illustrated in \cite{Hout1995,blanc1987numerical} and the references therein, this property is valid for among others any quasi birth-and-death system, the JSQ system, the coupled processor system, covering also the model under consideration.
	\end{property}
	
	In what follows, we apply PSA to the model at hand and  derive a recursive, computational scheme for the equilibrium joint queue length distribution. For more details on PSA, the interested reader is referred to \cite{blanc1987numerical, Dorsman2013, Hout1995, hooghiemstra1988power}, where one can find the detailed implementation of PSA for a plethora of systems. Following the steps of PSA in \cite{blanc1992}, we obtain and solve a recursive set of equations for the coefficients $\beta(n, k, l)$ in \eqref{PGF}. From this, the equilibrium distribution is computed, as well as any performance measures derived from it. 
	
	In Appendix  \ref{ApendixPSA}, we present all details of the method. For illustration purposes, we have chosen $a=\f{1}{2}$, so as to simplify the notation and enhance the readability of the method. The computation scheme of the PSA is summarised in the next paragraphs.
	
	We first substitute the power-series expansion \eqref{PGF} into the balance equations \eqref{in}-\eqref{RB3} and the normalisation equation. This leads to a polynomial expression in $\rho$ equal to zero, which we use to equate the corresponding powers of $\rho$, and obtain a recursion in the coefficients $\beta(n, k, l)$, $n \geq 0,(k,l)\in\hat{S}$. This results in a computational scheme through which we can compute the coefficients $\beta(n, k, l)$. The computation of the coefficients $\beta(n, k, l)$ is equivalent, in terms of mathematical complexity, to solving a system of equations, cf. Equations \eqref{CB1}-\eqref{CB7}, for a finite large value of $n \geq 0$.

	Having computed the coefficients $\beta(n, k, l)$, we use them to approximate, to any degree of required precision, the equilibrium distribution, using \eqref{PGF}, and thus, we can compute many performance measures by writing them as a power series in $\rho$. From a theoretical perspective, PSA can be used to compute the performance measures in a symbolic fashion, however, in practice, this can be  achieved only for coefficients of very small order. This is mainly due to the fact that it is required to solve symbolically a recursive  scheme, which is increasingly hard as the order of $\rho$ increases. The numerical computation of the performance measures, for given values of the system parameters, is possible in principle up to an arbitrary precision, by applying the recursion until the desired precision is achieved. However, note that the given values of the system parameters need to be chosen so as to be within the radius of convergence of \eqref{PGF}. In order to expand the range of values (within the stability region), for which PSA can produce accurate numerical results, Keane et al. \cite{hooghiemstra1988power} propose the application of a bilinear conformal transformation of the real interval $[0, 1)$ onto itself
	\begin{equation}\label{ConformalMapping}
	\theta =\frac{(1+G)\rho}{1+G\rho} \iff \rho = \f{\theta}{1 + G- G \theta},
	\end{equation} 
	with $G \geq 0$. Using the above conformal mapping, the power series \eqref{PGF} is written as a power series over the parameter $\theta$
	\begin{equation}\label{PGF_new1}
	\pi_{k, l} =  \sum\limits_{n = 0}^{\infty} ~\theta^{n + k + l}~ u(n, k, l),\ k,l\geq0.
	\end{equation} 
	The derivation of the recurrence relations for the coefficients $u(n, k, l)$ of the new power series \eqref{PGF_new1} is similar to the process for $\beta(n, k, l)$, see Appendix \ref{ApendixPSA}.

	\subsection{Numerical implementation of PSA} 
	In order to compare, the two approaches, namely the compensation approach and PSA, we sketch below the algorithmic implementation of the latter. Similarly to the compensation approach, we truncate the series expression \eqref{PGF_new1} to the $N_{\mathrm{psa}}$-th term. 
	
	\begin{algorithm}[H]		
		\caption{PSA implementation}\label{PSA_New}
		\begin{algorithmic}	[1]	
			\State Inputs $\l, a$, $G$ and precision $\e$.
			\State Set $N_{\mathrm{psa}} = 1$ and $u(0, 0, 0) = 1$.
			\State \label{Algo_PSA_Step30} Set $T_{\mathrm{PSA}} = \max\{ \lceil \log(\e)/\log(\rho^2) \rceil,3\}$
			\State  \label{Algo_PSA_Step3}For all $0 \leq k, l \leq T_{\mathrm{PSA}},\ 0\leq n \leq N_{\mathrm{psa}}$, compute the coefficients $u(n, k, l)$ by solving the system of linear equations cf.   \eqref{CBN1}-\eqref{CBN7}, in Appendix \ref{ApendixPSA}.
			\State \label{Algo_PSA_Step5} For all $0\leq k, l \leq T_{\mathrm{psa}}$, compute  $\pi^{(N_{\mathrm{psa}})}_{k, l} =  \sum\limits_{n = 0}^{N_{\mathrm{psa}}} ~\theta^{n + k + l}~ u(n, k, l)$.
			\State  Stop if {$\left| \frac{\sum\limits_{k,l = 0}^{T_{\mathrm{psa}}} \pi^{{(N_{\mathrm{psa}})}}_{k, l} - \sum\limits_{k,l = 0}^{T_{\mathrm{psa}}} \pi^{({N_{\mathrm{psa}}}-1)}_{k, l}}{\sum\limits_{k,l = 0}^{T_{\mathrm{psa}}} \pi^{{(N_{\mathrm{psa}})}}_{k,l}}\right| < \e$,} else update  $N_{\mathrm{psa}} = N_{\mathrm{psa}}+1$ and $T_{\mathrm{psa}} = T_{\mathrm{psa}}+1$ and go to Step \ref{Algo_PSA_Step3}.
		\end{algorithmic}		
	\end{algorithm}

	
	\section{Equilibrium analysis: The probability generating function approach}\label{gen}
	In the following, we show how one can apply the probability generating function (PGF) approach to analyse the model at hand. The analysis through the PGF approach leads to a functional equation, whose solution usually presents formidable difficulties. However, for the two-dimensional case, techniques have been developed to reduce the problem of the solution of the functional equation to standard problems from the theory of boundary value problems, see, e.g., \cite{fayolle,cohenboxma}. Even in these cases, the analysis is lengthy and complicated, involving sophisticated complex analysis, Riemann surfaces, and the determination of some conformal mapping. In most cases, this requires numerical analysis, which makes the formal solutions to the boundary value problems less insightful. This drawback is overcome by the use of the compensation approach and the PSA method developed in the previous sections, revealing their superiority with respect to the PGF approach. 
	
	Our aim hereon is to provide the basic steps of the analysis. To this purpose, we consider the original process $\{\bm{Q}(n),\, n \geq 0\}$, cf. Section \ref{Themodel}. Similarly to the analysis of the classical  JSQ model, we need to account for the  different transition patterns above and below the diagonal of the first quadrant, which can be seen as ``the third boundary", see, e.g., \cite[Chapter III]{cohenboxma}, \cite[Chapter 10]{fayolle}, and \cite{kurk2003}. Let
	\begin{displaymath}
	w_{i,j}=\lim_{n\to\infty}\mathbb{P}(Q_{1}(n)=i,Q_{2}(n)=j),\ i,j\geq0,
	\end{displaymath}
	denote the equilibrium joint queue length distribution of the original random walk $\{\bm{Q}(n),\, n \geq 0\}$. From the balance equations, we obtain after tedious, but straightforward algebra
	\begin{eqnarray}
	&&T(x,y)\mathbb{E}(x^{Q_{1}}y^{Q_{2}}(Q_{1}<Q_{2}))+T(y,x)\mathbb{E}(x^{Q_{1}}y^{Q_{2}}(Q_{1}>Q_{2}))+F(x,y)\mathbb{E}(x^{Q_{1}}y^{Q_{2}}(Q_{1}=Q_{2}))\nonumber\\
	&&\quad +C(x,y)\mathbb{E}(x^{Q_{1}}(Q_{2}=0))+C(y,x)\mathbb{E}(y^{Q_{2}}(Q_{1}=0))=w_{0,0}L(x,y),~~ |x|, |y| \leq 1, 
	\label{fun}
	\end{eqnarray}
	with
	\begin{align*}
	T(x,y)=&\lambda(\bar{a}^{2}+a^{2})(1-x)+\bar{\lambda}\bar{a}a(2-\frac{1}{x}-\frac{1}{y})+\lambda\bar{a}a(1-\frac{x}{y}),\\
	F(x,y)=&\frac{\lambda}{2}(2-x-y)(\bar{a}^{2}+a^{2})+\bar{\lambda}\bar{a}a(2-\frac{1}{x}-\frac{1}{y})+\frac{\lambda}{2}\bar{a}a(2-\frac{y}{x}-\frac{x}{y}),\\
	C(x,y)=&\bar{\lambda}a[a(1-\frac{1}{x})-\bar{a}(1-\frac{1}{y})],\\
	L(x,y)=&\frac{\lambda}{2}a(a-\bar{a})(2-x-y)+\bar{\lambda}a^{2}(2-\frac{1}{x}-\frac{1}{y})+\frac{\lambda}{2}a\bar{a}(2-\frac{x}{y}-\frac{y}{x}).
	\end{align*}
	The main steps of the approach are summarised below and follow considerably the lines in \cite{cohenboxma}.
	\begin{description}
		\item[Step 1] We employ an idea similar to the Wiener-Hopf factorisation\footnote{Note that a similar functional equation was investigated in \cite[Chapter III]{cohenboxma},  \cite[Chapter 10]{fayolle}. However, in the case under consideration due to the special behaviour of the random walk at the boundaries, and due to the existence of the South-East transitions (see Fig. \ref{rw}, rays $V$, $V^{\prime}$, $D$), and the North-West transitions (see Fig. \ref{rw}, rays $H$, $H^{\prime}$, $D$), the resulting functional Equation \eqref{fun} is essentially different, which considerably complicates the analysis.}; see \cite[Chapter III]{cohenboxma}, and introduce an appropriate transformation of $x$, $y$, which leads to two new variables. Then, we fix one of the two variables and define a smooth closed contour in the other variable, say the free variable. Let $x=\rho_{1}/u$, $y=\rho_{2}u$. Then \eqref{fun} is rewritten as
		\begin{align}
		&
		\mathbb{E}(\rho_{1}^{Q_{1}}\rho_{2}^{Q_{2}}u^{Q_{2}-Q_{1}}(Q_{2}>Q_{1}))T(\rho_{1}/u,\rho_{2}u)u^{2}+\mathbb{E}((\rho_{2}u)^{Q_{2}}(Q_{1}=0))C_{1}(\rho_{1}/u,\rho_{2}u)u^{2} \nonumber \\
		&
		+\mathbb{E}(\rho_{1}^{Q_{1}}\rho_{2}^{Q_{2}}(Q_{2}=Q_{1}))F(\rho_{1}/u,\rho_{2}u)u^{2}-w_{00}J(\rho_{1}/u,\rho_{2}u)u^{2} \nonumber \\
		=&-\Big(\mathbb{E}(\rho_{1}^{Q_{1}}\rho_{2}^{Q_{2}}u^{Q_{2}-Q_{1}}(Q_{2}<Q_{1}))T(\rho_{2}u,\rho_{1}/u)u^{2} + \mathbb{E}((\rho_{1}^{Q_{1}}u^{-Q_1}(Q_{2}=0))C_{2}(\rho_{1}/u,\rho_{2}u)u^{2} \nonumber \\
		&\qquad + 
		\mathbb{E}(\rho_{1}^{Q_{1}}\rho_{2}^{Q_{2}}(Q_{2}=Q_{1}))F(\rho_{2}u,\rho_{1}/u)u^{2}-w_{00}J(\rho_{2}u,\rho_{1}/u)u^{2}\Big),
		\label{cv1}
		\end{align}
		with
		\begin{displaymath}
		J(x,y)=\bar{\lambda}a^{2}(1-\frac{1}{y})+\frac{\lambda}{2}\bar{a}a(1-\frac{x}{y})+\frac{\lambda}{2}a(a-\bar{a})(1-x).
		\end{displaymath}
		\item[Step 2] Note that for $|\rho_{1}|\leq 1$, $|\rho_{2}|\leq 1$, the left-hand side of \eqref{cv1} is regular for $|u|<1$, and continuous for $|u|\leq 1$. Similarly, the right-hand side of \eqref{cv1} is regular for $|u|>1$, and continuous for $|u|\geq 1$, and as $|u|\to\infty$, it behaves as $|u|^{4}$. By considering the series expansion in powers of $u$, $u^{-1}$ of all the terms appearing in \eqref{cv1}, and equating the coefficients of equal powers, Liouville's Theorem will immediately reveal that the partial bivariate PGFs, $\mathbb{E}(x^{Q_{1}}y^{Q_{2}}(Q_{1}<Q_{2}))$, $\mathbb{E}(x^{Q_{1}}y^{Q_{2}}(Q_{1}>Q_{2}))$ behave as constants in the free variable. Now dividing \eqref{cv1} by $u^{2}$, and taking $\rho_{1}=r_{1}u$, $\rho_{2}=r_{2}/u$, $|u|=1$, yields that, for $|r_{1}|\leq 1$, $|r_{2}|\leq 1$,
		\begin{align}
		&T(r_{1},r_{2})\mathbb{E}(r_{1}^{Q_{1}}r_{2}^{Q_{2}}(Q_{1}<Q_{2}))
		+F(r_{1},r_{2})\Phi_{0}(r_{1}r_{2})+C(r_{1},r_{2})\mathbb{E}(r_{2}^{Q_{2}}(Q_{1}=0))-w_{0,0}J(r_{1},r_{2})=0,\label{fun11}\\
		&T(r_{2},r_{1})\mathbb{E}(r_{1}^{Q_{1}}r_{2}^{Q_{2}}(Q_{1}>Q_{2}))
		+F(r_{2},r_{1})\Phi_{0}(r_{1}r_{2})+C(r_{2},r_{1})\mathbb{E}(r_{1}^{Q_{1}}(Q_{2}=0))-w_{0,0}J(r_{2},r_{1})=0,\label{fun12}
		\end{align}
		with $\Phi_{0}(r_{1}r_{2}):=\mathbb{E}(r_{1}^{Q_{1}}r_{2}^{Q_{2}}(Q_{2}=Q_{1}))$. 
		
		\item[Step 3] We  next investigate the tuples of the kernel $T(r_{1},r_{2})=0$. In Lemma \ref{lemgf1}, we show that for $|r_{2}|=1$, $r_{2}\neq1$, $T(r_{1},r_{2})=0$ has two roots, say $\delta_{0}(r_{2})$, $\delta_{1}(r_{2})$, such that $|\delta_{0}(r_{2})|<1<|\delta_{1}(r_{2})|$.
		
		Denote by $r_{2}^{(0)}<r_{2}^{(1)}$ the branching points of $T(\delta(r_{2}),r_{2})=0$, i.e., the zeros of the discriminant of \eqref{mk}, and let the slit $\mathcal{G}_1=\{r_{2}\in\mathbb{C}:r_{2}\in[r_{2}^{(0)},r_{2}^{(1)}]\}$. From \cite{fayolle}, Lemma 2.3.8, pp. 26-27, both branching points are real and located inside the unit disk. At $\mathcal{G}_1$, the two branches of $\delta(r_{2})$, i.e., $\delta_{0}(r_{2})$ and $\delta_{1}(r_{2})$, are complex conjugates. Denote the image contours, $\mathcal{M}:=\delta_{0}[\overrightarrow{\underleftarrow{r_{2}^{(0)},r_{2}^{(1)}}}]$, where $[\overrightarrow{\underleftarrow{u,v}}]$ stands for the contour traversed from $u$ to $v$ along the upper edge of the slit $[u,v]$ and then back to $u$ along its lower edge. Analogously, we can define the slit $\mathcal{G}_{2}=\{r_{1}\in\mathbb{C}:r_{1}\in[r_{1}^{(0)},r_{1}^{(1)}]\}$, where $r_{1}^{(0)}<r_{1}^{(1)}$ the zeros of the discriminant of $T(r_{2},r_{1})=0$.  Note that due to symmetry, $r_{j}^{(0)}=r_{j}^{(1)}$, $j=1,2$, and $\mathcal{G}_{1}\equiv \mathcal{G}_{2}$. In Lemma \ref{lemgf2}, we provide the exact representation of $\mathcal{M}$, while in Lemma \ref{lemgf3}, we show that $\delta_{1}(r_{2})$ is analytic in $\mathbb{C} \smallsetminus [r_{2}^{(0)},r_{2}^{(1)}]$.

		
		\item[Step 4] We proceed with the formulation of the boundary value problem. Note that
		\begin{displaymath}
		\mathbb{E}(r_{1}^{Q_{1}}r_{2}^{Q_{2}}(Q_{1}<Q_{2}))=\mathbb{E}((r_{1}r_{2})^{Q_{1}}r_{2}^{Q_{2}-Q_{1}}(Q_{1}<Q_{2}))
		\end{displaymath}
		is regular for $|r_{1}|<|\frac{1}{r_{2}}|$, for fixed $r_{2}$ with $|r_{2}|\leq 1$. Taking into account that $T(r_{1},r_{2})=0$, then, for $|r_{2}|<1$, $\mathcal{R}\mathit{e}(1-\frac{1}{r_{2}})\geq 0$, Equation \eqref{fun11} yields
		\begin{align}
		\mathbb{E}(r_{2}^{Q_{2}}(Q_{1}=0))=
		&\left (-\frac{(r_{2}-r_{1})\Phi_{0}(r_{1}r_{2})}{2[r_{1}r_{2}(q-1)-qr_{2}+r_{1}]}\right.\nonumber\\
		&+\left.w_{0,0}\frac{r_{1}\left[2\bar{\lambda}a^{2}(r_{2}-1)+\lambda\bar{a}[a(r_{2}-r_{1})+(q-1)r_{2}(1-r_{1})]\right]}{2\bar{\lambda}a\bar{a}[r_{1}r_{2}(q-1)-qr_{2}+r_{1}]}\right)\Bigg|_{r_{1}=\delta_{0}(r_{2})},
		\label{fun13}
		\end{align}
		where $q=a/\bar{a}$. Similarly, from \eqref{fun12}, taking into account that $T(r_{2},r_{1})=0$, then, for $|r_{1}|<1$, $\mathcal{R}\mathit{e}(1-\frac{1}{r_{1}})\geq 0$,
		\begin{align}
		\mathbb{E}(r_{1}^{Q_{1}}(Q_{2}=0))=
		&\left(-\frac{(r_{1}-r_{2})\Phi_{0}(r_{1}r_{2})}{2[r_{1}r_{2}(q-1)-qr_{1}+r_{2}]}\right.\nonumber\\
		&+\left.w_{0,0}\frac{r_{2}\left[2\bar{\lambda}a^{2}(r_{1}-1)+\lambda\bar{a}[a(r_{1}-r_{2})+(q-1)r_{1}(1-r_{2})]\right]}{2\bar{\lambda}a\bar{a}[r_{1}r_{2}(q-1)-qr_{1}+r_{2}]}\right)\Bigg|_{r_{2}=\delta_{0}(r_{1})}.
		\label{fun14}
		\end{align}
		For $|r|<1$, $\Phi_{0}(r)$ is regular, and $\mathbb{E}(r_{2}^{Q_{2}}(Q_{1}=0))$ is regular for $|r_{2}|<1$. Thus, the right-hand side in \eqref{fun13} can be continued analytically into $|r_{2}|<1$, $\mathcal{R}\mathit{e}(1-\frac{1}{r_{2}})\leq 0$. Since $\Phi_{0}(r)$ is well defined for $|r|<1$ and $\mathcal{R}\mathit{e}(1-\frac{1}{r})\leq 0$, we conclude that $\Phi_{0}(\delta_{1}(r_{2})r_{2})$ can be analytically continued into $|r_{2}|<1$, $\mathcal{R}\mathit{e}(1-\frac{1}{r_{2}})\leq 0$. Since $\mathbb{E}(r_{2}^{Q_{2}}(Q_{1}=0))$ is real for $r_{2}\in \mathcal{G}_{1}$, Equation \eqref{fun13} yields for $r_{2}\in \mathcal{G}_{1}$,
		\begin{equation}\label{bv1}
		\begin{array}{l}
		\mathcal{I}\mathit{m}\left[v(r_{1},r_{2})\Phi_{0}(r_{1}r_{2})\right]\bigg\rvert_{r_{1}=\delta_{0}(r_{2})}=\mathcal{I}\mathit{m}\left[w_{0,0}f(r_{1},r_{2})\right]\bigg\rvert_{r_{1}=\delta_{0}(r_{2})},
		\end{array}
		\end{equation}
		Similarly, for $r_{1}\in \mathcal{G}_{2}$,
		\begin{equation}\label{bv2}
		\begin{array}{l}
		\mathcal{I}\mathit{m}\left[v(r_{2},r_{1})\Phi_{0}(r_{1}r_{2})\right]\bigg\rvert_{r_{2}=\delta_{0}(r_{1})}=\mathcal{I}\mathit{m}\left[w_{0,0}f(r_{2},r_{1})\right]\bigg\rvert_{r_{2}=\delta_{0}(r_{1})}.
		\end{array}
		\end{equation}
		where $v(r_{1},r_{2})=\frac{(r_{2}-r_{1})}{2[r_{1}r_{2}(q-1)-qr_{2}+r_{1}]}$, $f(r_{1},r_{2})=\frac{r_{1}\left[2\bar{\lambda}a^{2}(r_{2}-1)+\lambda\bar{a}[a(r_{2}-r_{1})+(q-1)r_{2}(1-r_{1})]\right]}{2\bar{\lambda}a\bar{a}[r_{1}r_{2}(q-1)-qr_{2}+r_{1}]}$.
		
		From the discussion so far, and due to the symmetry, it is natural to consider in what follows only one of the above boundary conditions, say \eqref{bv1}. Note that $\mathbb{E}(r_{2}^{Q_{2}}(Q_{1}=0))$ can be obtained by \eqref{fun13} upon deriving $\Phi_{0}(\cdot)$. Then, having this information, $\mathbb{E}(r_{1}^{Q_{1}}r_{2}^{Q_{2}}(Q_{1}<Q_{2}))$ is obtained by \eqref{fun11}. Similarly, we derive $\mathbb{E}(r_{1}^{Q_{1}}(Q_{2}=0))$ by \eqref{fun14}. Then, using \eqref{fun12}, we finally obtain $\mathbb{E}(r_{1}^{Q_{1}}r_{2}^{Q_{2}}(Q_{1}>Q_{2}))$. As a consequence, by obtaining the \textit{key element} $\Phi_{0}(\cdot)$, we are able to  derive $$\mathbb{E}(r_{1}^{Q_{1}}r_{2}^{Q_{2}}):=\mathbb{E}(r_{1}^{Q_{1}}r_{2}^{Q_{2}}(Q_{1}>Q_{2}))+\mathbb{E}(r_{1}^{Q_{1}}r_{2}^{Q_{2}}(Q_{1}<Q_{2}))+\mathbb{E}(r_{1}^{Q_{1}}r_{2}^{Q_{2}}(Q_{1}=Q_{2})).$$
		
		From Lemma \ref{lemgf2}, we know that when $r_{2}\in G_{1}$, then $\delta_{0}(r_{2})\in\mathcal{M}$; following \cite[Sect. 10.3, p. 204]{fayolle},  $\mathcal{M}$ is an ellipse. To proceed, we introduce the mapping $\mathcal{H}:\ r_{2}\to z,\,z=x+iy=r_{2}\delta_{0}(r_{2})$, $r_{2}\in \mathcal{G}_{1}$, and set $\mathcal{E}:=\mathcal{H}(\mathcal{G}_{1})$; \cite{cohenboxma}. Then, \eqref{bv1} is reduced to
		\begin{equation}\label{bvv}
		\mathcal{R}\mathit{e}\left[i\mathcal{\beta}(z)\Phi_{0}(z)\right]=\mathcal{\eta}(z),\,z\in\mathcal{E},
		\end{equation}
		where $\mathcal{\beta}(z)$, $\mathcal{\eta}(z)$ are the translations of $v(\cdot)$ and the right-hand side of \eqref{bv1} under the mapping $\mathcal{H}$\footnote{Some discussion about the possible poles of $\Phi_{0}(\cdot)$, i.e., the zeros of $\beta(\cdot)$ in $\mathcal{E}^{+}\cap\mathcal{D}^{c}$, where $\mathcal{D}=\{t\in\mathbb{C}:|t|\leq1\}$, and $\mathcal{G}^{+}$ in the interior domain bounded by the contour $\mathcal{G}$ must be made. To enhance the readability, we assume hereon that there are no such zeros. In any case, the analysis can be generalised directly.}. The usual procedure is to consider this problem to the unit circle $\mathcal{C}$. In general, it is hard to explicitly construct the conformal mappings. However, there is an efficient way to numerically obtain them via Theodorsen's procedure\footnote{Alternatively, one can follow \cite{neh} and define a conformal mapping using Jacobi elliptic functions.}; for further details see \cite{cohenboxma}, Sect. IV.1.3. Let the conformal mapping $t=\theta_{0}(z):\mathcal{E}^{+}\to\mathcal{C}^{+}$, and its inverse $z=\theta_{1}(t):\mathcal{C}^{+}\to\mathcal{E}^{+}$. Then, \eqref{bvv} is reduced to: Find a function $\Omega(t):=\Phi_{0}(\theta_{1}(t))$ regular for $|t|<1$, continuous for $|t|\leq 1$ such that
		\begin{equation}\label{bvv1}
		\mathcal{R}\mathit{e}\left[i\mathcal{\beta}(\theta_{1}(t))\Omega(t)\right]=\mathcal{\eta}(\theta_{1}(t)),\,|t|=1.
		\end{equation}
		The following theorem provides an expression of the \textit{key element} $\Phi_{0}(\cdot)$ in terms of the general solution of a non-homogeneous Riemann-Hilbert boundary value problem \cite[Chapter IV]{ga}.
		\begin{theorem}
			For $\rho<1$, the key element $\Phi_{0}(\cdot)$ is the solution of the non-homogeneous Riemann-Hilbert problem on $\mathcal{C}$ defined in \eqref{bvv1}, and given by (cf. \cite[Section 29.3]{ga})
			\begin{equation}
			\Phi_{0}(\theta_{1}(t))=e^{i\sigma(t)}t^{\chi}\left[\frac{1}{2\pi i}\int_{|u|=1}e^{\omega_{1}(u)}\phi(u)\frac{u+t}{u-t}\frac{du}{u}+iK\right],\,|t|<1,
			\end{equation}
			where 
			$\chi=-\frac{1}{\pi}arg\{\mathcal{\beta}(z)\}_{z\in\mathcal{E}}$ is the index of the problem, 
			$K$ is a constant to be determined, 
			$\sigma(t)=\frac{1}{2\pi i}\int_{|u|=1}(\arctan\frac{\mathcal{R}\mathit{e}(\mathcal{\beta}(\theta_{1}(t)))}{\mathcal{I}\mathit{m}(\mathcal{\beta}(\theta_{1}(t)))}
			-\chi\arg u)\frac{t+u}{t-u}\frac{du}{u}$, 
			$\omega_{1}(t)=\mathcal{I}\mathit{m}(\sigma(t))$, 
			$\phi(t)=\frac{\mathcal{\eta}(\theta_{1}(t))}{|\mathcal{\beta}(\theta_{1}(t))|}$.
			
			If $\chi\leq 0$ there is at most one linearly independent solution. When $\chi=0$, $K$ can be determined from the value of $\Phi_{0}(\cdot)$ at the origin. If $\chi<0$, then $K=0$ and a solution exists if $\frac{1}{2\pi i}\int_{|u|=1}e^{\omega_{1}(u)}\phi(u)u^{-k-1}du=0$ for $k=0,1,...,-\chi-1.$ Note that $\Phi_{0}(z)=\Phi_{0}(\theta_{1}(\theta_{0}(z))))$.
		\end{theorem}
	\end{description}

	\section{Comparison of methods}\label{Sec:num}
	In this section, we compare the compensation approach and PSA method on the basis of their performance for the system at hand. The comparison is performed in terms of numerical accuracy (absolute difference in the obtained results by both methods) and time complexity. Algorithms \ref{Compensation_Algo} and \ref{PSA_New} are designed so as to compute any performance measure, which depends on the equilibrium distribution, given a desired precision $\e$.
	
	\subsection{Comparison of numerical accuracy}
	With the compensation approach, we can compute an explicit expression of the equilibrium distribution as a (infinite) linear combination of product-form terms. We have proven, that the corresponding truncated linear combinations provide asymptotic expansions which improve as we include more terms. In this perspective, we can relatively control the error of the approach. However, for PSA, to the best of our knowledge, there exists no error bound due to the lack of theoretical support for this approach, see \cite{blanc1987numerical, blanc1992}. The error produced by the PSA implementation can be controlled somewhat by including more terms of the series. On top of that,  it is sometimes unclear when this method diverges \cite{Koole1994}, but the radius of convergence of the power series can be extended using a transformation, cf. \eqref{ConformalMapping}. 
	
	In Table \ref{T4}, we depict the total expected sojourn time of a packet measured in time slots and the correlation coefficient  between the queue lengths of the two relays, for various values of the load $\rho$. 	The  total expected sojourn time of a packet  is computed as $\mathbb{E} [S] = \frac{1}{\lambda}\mathbb{E}[Q_1+Q_2]$, and as expected, as $\rho$ increases, so do the values of $\mathbb{E} [S]$. 
	The correlation coefficient between the queue relays is computed as $\mathbb{R}(Q_1, Q_2) = \f{\mathbb{E}[Q_1Q_2]- \mathbb{E}[Q_1] \mathbb{E}[Q_2]}{\sqrt{\mathbb{V}ar[Q_1] \mathbb{V}ar[Q_1]}}$. For the computations performed, we choose $\rho=0.1,0.4,0.7,0.9,0.95$, so as to cover lighter and heavier traffic results. Furthermore, we choose  $G = 1$ for the PSA implementation.
	
 As expected, as $\rho$ increases the correlation coefficient tends to one, and as the values of Table \ref{T4} indicate, it almost behaves like a linear function of $\rho$. Furthermore, it is evident, from the numerical results,  that both approaches produce similar outcomes (as long as the value for $\rho$ is away from the stability region boundary $\rho=1$) differing by approximately as much as the range of the precision, cf. column $\mid \mathrm{CA} -  \mathrm{PSA} \mid$. However, as $\rho$ approaches one, PSA becomes highly unstable, while the compensation approach is producing accurate results in the entire stability region. The numerical instability of  PSA can be explained observing that as $\rho\to1$, $\theta\to1$, cf. Equation \eqref{ConformalMapping}, indicating that the power series expression is approaching the boundary of the region of convergence. This can be overcome, to a certain degree, by further increasing the value $G\geq1$. 
	
	Note that the compensation approach only works in the case of two relays, while PSA can be extended to any finite number of relays. As such, we conclude that the compensation approach is more suitable when working with two relays, while  PSA is generalisable to a system with more than  two relays. 
	
	\begin{table}[h!]
		\begin{center}
			\scalebox{0.9}{
				\begin{tabular}{ c | c | c |c}
					~&  $\mathbb{E} [S]$ & $\mathbb{R}(Q_1,Q_2)$   &  \\ 
					\hline 
					\begin{tabular}{c  }
						$\rho$  \\
						\hline
						$0.1$ \\ $0.4$ \\ $0.7$ \\ $0.9$ \\ $0.95$ \end{tabular} &  \begin{tabular}{c | c | c}
						CA & PSA & $\mid \mathrm{CA} -  \mathrm{PSA} \mid$  \\
						\hline
						$1.222$ & 1.222 & $1.9 \times 10^{-14}$ \\		
						$2.333$ &  $2.333$  & $6.8 \times 10^{-8}$ \\	
						$5.666$ &  $5.666$ & $6.1 \times 10^{-4}$\\	
						$19.00$  & $17.858$ & $1.412$\\
						$38.976$ & $22.450$ & $16.525$
					\end{tabular} & 	\begin{tabular}{c | c |c }
						CA & PSA & $\mid \mathrm{CA} -  \mathrm{PSA} \mid$   \\
						\hline
						$0.136$ & $0.136$ & $5.6 \times 10^{-14}$ \\		
						$0.468$ & $0.468$ & $2.1 \times 10^{-7}$\\	
						$0.793$ & $0.792$ & $2.9 \times 10^{-4}$\\	
						$0.969$ & $0.950$ & $1.9 \times 10^{-2}$\\
						$0.991$ &  $0.924$ &  $6.7 \times 10^{-2}$
					\end{tabular} & \begin{tabular}{c  }~~~~~~~~ $\e$
						\textcolor{white}{$R_{\mathrm{error}}$} \\
						\hline
						$10^{-30}$ \\  $10^{-25}$ \\  $10^{-20}$ \\  $10^{-15}$ \\ $10^{-10}$  \end{tabular}	
			\end{tabular}}
			\caption{Total expected sojourn time of a packet and  the correlation between the queue relays.}
			\label{T4}
		\end{center}
	\end{table}

	\subsection{Comparison of time complexity}
	
	From the numerical implementation, it is notable that both approaches provide accurate results to a desired precision. In both approaches, the time complexity  of the corresponding algorithm depends on $\e$ through the determination of $T_{\mathrm{ca}}$ or $T_{\mathrm{psa}}$, cf. Step \ref{Algo_CA_Step4} or Step \ref{Algo_PSA_Step30}, respectively, which we can control. Equating $\e$ for both algorithms, we can compare the methods in terms of their algorithmic time complexity, so as to characterise the performance speed of the two approaches. We describe time complexity in the Big-$O$ notation as a function of the input size, i.e., $O(f(\cdot))$ is measured as the maximum number of elementary steps needed to perform the algorithm, provided that each step is executed in constant (or equal) time.  This way, the time required for the algorithm is described independently of the numerical implementation.  
	
	We discuss the time complexity of the two approaches separately: For the compensation approach, see Algorithm \ref{Compensation_Algo}, for a series truncation level $N_{\mathrm{ca}}$ and state-space truncation $T_{\mathrm{ca}}$, the algorithm  converges with order $O\left( N_{\mathrm{ca}}\left(  T_{\mathrm{ca}}\right)^{2.376}\right)$. This order is obtained as follows:
	\begin{enumerate}
		\item[(i)]  In Steps \ref{Algo_CA_Step2} and \ref{Algo_CA_Step5}, we need to compute recursively the $2(N_{\mathrm{ca}}+1)$ roots of Equation  \eqref{eq}. This can be done by using the Bisection method or the False position (aka regula falsi)  method. With both methods, we are able to choose the bisection intervals within the interval $(0, \g_0)$, which results in time complexity $O(\log(\g_0))$ for the computation of a single root. In order to compute all the roots, the Bisection method needs to be repeated at least $2(N_{\mathrm{ca}}+1)$ times. Thus, the time complexity for the calculation of  all the roots is of order  $O(N_{\mathrm{ca}} \log(\g_0))$. 
		
		\item[(ii)] In Step \ref{Algo_CA_Step6}, we need to compute recursively the coefficients $c_i, e_{i, o}, e_{i, 1}$ and $d_i$, $i=0,1,\ldots,N_{\mathrm{ca}}$. To this purpose, we need to solve a system of equations. The system is solved implementing the Coppersmith-Winograd algorithm \cite{Coppersmith1987}, which has a complexity $O(N_{\mathrm{ca}}^{2.376})$. 
		
		\item[(iii)] In Step \ref{Algo_CA_Step7}, we need to compute the equilibrium probabilities $\pi_{k, l}^{(N_{\mathrm{ca}})}$ using \eqref{on-numer} and \eqref{on2-numer}, which has time complexity $O(N_{\mathrm{ca}} (T_{\mathrm{ca}}/2))$.
		
		\item[(iv)]  In Step \ref{Algo_CA_Step8}, we need to solve a system of equations. Using again the Coppersmith-Winograd algorithm. This step  has a complexity of $O( (T_{\mathrm{ca}}/2)^{2.376})$, and needs to be performed at least $N_{\mathrm{ca}}$ times. Thus, the time complexity of the entire step is of order $O\left( N_{\mathrm{ca}}\left(  T_{\mathrm{ca}} /2\right)^{2.376}\right)$. 
		
		\item[(v)]  In the construction of Algorithm \ref{Compensation_Algo}, we have considered a part of the state-space region in which we directly use  Equation \eqref{on2-numer}, cf. Step \ref{Algo_CA_Step6}, and a part of the state-space region in which we solve a linear system of the  equations cf. Step \ref{Algo_CA_Step7}. The sizes of these regions can be at most equal to $N_{\mathrm{ca}}$. Thus in the worst case this results in Step \ref{Algo_CA_Step7} requiring $O(N_{\mathrm{ca}} T_{\mathrm{ca}})$ and  Step \ref{Algo_CA_Step8} requiring $O\left( N_{\mathrm{ca}}\left(  T_{\mathrm{ca}}\right)^{2.376}\right)$, respectively. Note that (iv) has the dominant time complexity, and under the worst case scenario, it yields a time complexity of $O\left( N_{\mathrm{ca}}\left(  T_{\mathrm{ca}}\right)^{2.376}\right)$.
	\end{enumerate}

Analogously,  for the power series algorithm, see Algorithm \ref{PSA_New}, for a series truncation level $N_{\mathrm{psa}}$ and state-space truncation $T_{\mathrm{psa}}$, the algorithm  converges with order $O\left((N_{\mathrm{psa}}+1)\left((T_{\mathrm{psa}}+1)^2/4\right)^{2.376}\right)$. This order is obtained as follows:
	
\begin{enumerate}
	\item[(i)]	In Step \ref{Algo_PSA_Step5}, for $0 \leq n \leq N_{\mathrm{psa}}$ and $0\leq k, l \leq T_{\mathrm{psa}}$, we compute  $u(n, k, l)$. This step determines the dominant time complexity in PSA.  These coefficients are obtained solving a system of equations \eqref{CBN1}-\eqref{CBN7}, in Appendix \ref{ApendixPSA}, by  implementing for example the Coppersmith-Winograd algorithm \cite{Coppersmith1987}. Note that this system of equations reveals that (i) the series truncation level should be chosen such that $N_{\mathrm{psa}} \leq T_{\mathrm{psa}}$, and (ii) the state-space truncation should contain all states $0\leq k+l \leq T_{\mathrm{psa}}$. All in all, this yields that the complexity is $O\left((N_{\mathrm{psa}}+1)\left((T_{\mathrm{psa}}+1)^2/4\right)^{2.376}\right)$.
		\end{enumerate}
	
	In the above discussion, we have demonstarted that the compensation algorithm has better big-$O$ time complexity than PSA.

	\subsection{Comparison of JSRQ to other routing protocols}
	The JSRQ routing protocol, balances  the load between the two relays, but due to this balancing, it also seems to increase collisions. For this reason, in this section, we compare the JSRQ routing protocol to a single server system. To make the two systems comparable, we consider the arrival and departure probabilities of the single server system to be the same as in the JSRQ system,  i.e., $\l$ and $a$. Furthermore, in order to identify the comparable region for both models we use the corresponding stability conditions. The single server system is stable if $\lambda < a$,  which implies $a \in (\l, 1)$. The JSQR system is stable for $\lambda < 2a(1-a)$. The stability region of the JSQR system can be equivalently written as 
	\begin{equation} 
	a \in (a^{-}, a^{+}) \equiv \left( \f{1-\sqrt{1-2\l}}{2}, \f{1+\sqrt{1-2\l}}{2}\right), \text{ for  }\l < 1/2.\label{Eq: JSQR_a_pm}
	\end{equation}
	The comparison between the protocols is discussed on the basis of the following points:
	
	\begin{enumerate}
		\item[(i)]\textbf{stability region comparison:} From the above discussion, it becomes evident that the length of the stability region of the single server system is $1-\l$ and of the JSRQ system it is $\sqrt{1-2\l}$. For $\l < 1/2$, we observe that the region of the single server system contains the region of the JSQR system.

		\item[(ii)]\textbf{Total expected queue length comparison:} For the single server system, the total expected queue length goes to infinity as $a\downarrow\lambda$  and to zero as $a\uparrow1$. For the JSQR system, the total expected queue length goes to infinity as $a\to a^{+}$ or $a\to a^{-}$, defined in Equation \eqref{Eq: JSQR_a_pm}. Furthermore, for $\lambda<1/2$, it is easy to show that $a^-<\lambda<1/2<a^+$. Moreover, the two systems have identical total expected queue lengths if $a=1/2$. Note that for $a=1/2$ the load of the single server system, $\lambda \bar{a}/\bar{\lambda}a$, is equal to the load of the JSRQ system, $\lambda(\bar{a}^{2}+a^{2})/2\bar{\lambda}\bar{a}a$. All in all, the above analysis reveals that for $a<1/2$ the JSRQ outperforms the single server system, for $a=1/2$ the systems perform identically, while for $a>1/2$ the single server system performs better than the JSRQ system. See Figure \ref{Comparision_JSRQ_Singleserver}, for a depiction of the above general remarks.
		
		\begin{figure}[hhh!]
			\begin{center}
				\includegraphics[scale= 0.35]{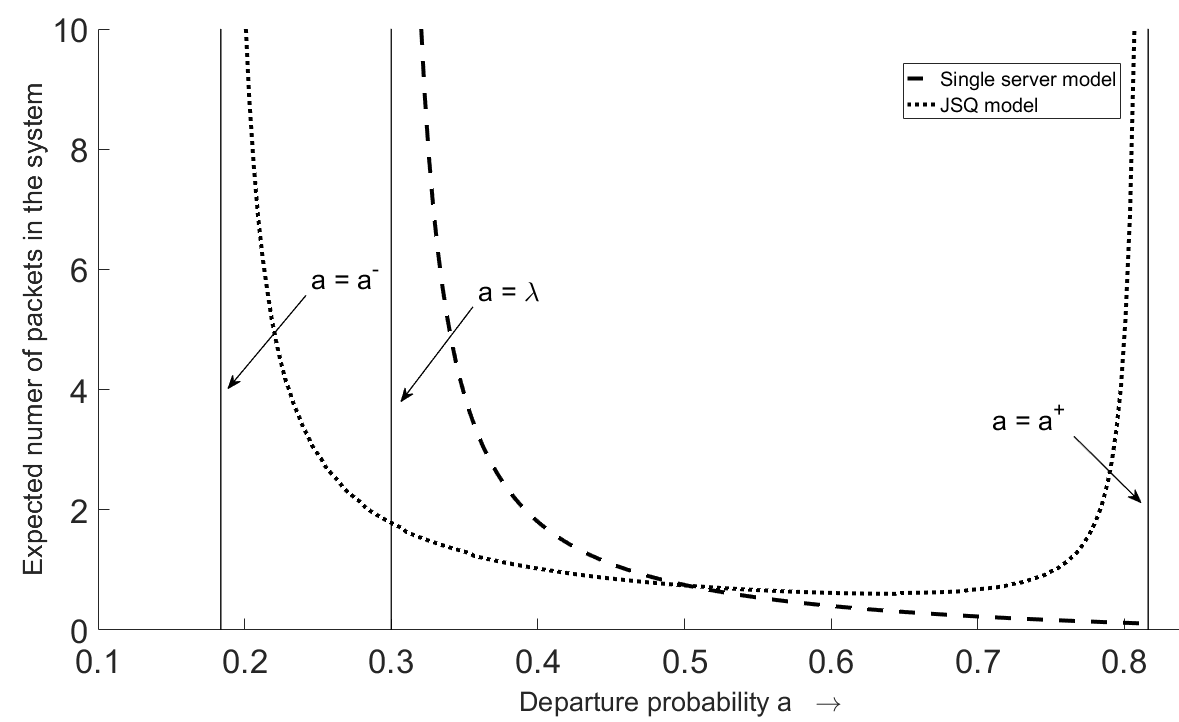}
			\end{center}	
			\caption{The expected number of packets $\mathbb{E}[Q_1+Q_2]$  are depicted on the $y$-axis as a function of the service probability $a$ (depicted on the $x$-axis) for fixed arrival probability $\l = 0.3$}
			\label{Comparision_JSRQ_Singleserver}
		\end{figure}
		
		\item[(iii)]\textbf{Consideration of collision:} Our model incorporates the phenomenon of collisions, whereas the single server model does not consider any collisions between the packets. Hence in the single server model the probability of collisions between the packets is zero, whereas in the considered model when both relays transmit packets in the same slot, the packets need to be retransmitted in a later slot. Communication networks in practice adopt the technique of retransmission in case of collisions to avoid losses, hence the considered model represents reality more closely than a single server model.
	\end{enumerate}

	\section{Conclusions and possible extensions}\label{Sec:con}
	In this work, we focused on the application oriented problem of characterising the queueing delay experienced in a slotted-time relay-assisted cooperative random access wireless network with collisions and join the shortest relay queue (JSRQ) routing protocol. Note that due to the collisions, there is strong interdependence among the queues of the relays, resulting in different service probabilities, when both relays are busy, than the service probability, when one of the relays is empty. Thus, the system at hand incorporates two features: the JSRQ feature and the coupled processor feature. For this system, we investigate the stability condition and apply three different methods for the computation of the equilibrium joint queue (relay) length distribution, namely the compensation approach , the power series algorithm (PSA), and the probability generating function approach. A detailed comparison of the compensation approach and PSA is presented. More importantly, we have applied the compensation approach to a random walk on the positive quadrant with transitions to a bounded region of neighbours, extending the framework of the compensation approach to a wider class of random walks (than the nearest neighbour for which the approach was originally developed).
	
	In a future work, we plan to generalise this work in several directions. A challenging task is related to the equilibrium analysis of a cooperative network with a queue-aware transmission protocol under which each relay configures its transmission parameters based on the status of the other. Such a protocol serves towards self-aware and intelligent networks. Moreover, we also plan to characterise the delay using a multi-packet reception model instead of the collision channel model. Under such a scheme, we can have successful transmissions even if multiple nodes (relays or source(s)) transmit simultaneously, which will definitely improve the throughput performance of the network. An interesting challenging task is the investigation of the queueing delay at a random access network with an arbitrary number of relay nodes under the JSRQ routing policy. Finally, it will be interesting to investigate service policies by taking into account the state of the network, for the ultimate goal of improving the system performance.

	\section{Acknowledgments}
	{\small The authors gratefully acknowledge useful discussions with Onno Boxma. The research of Mayank Saxena is supported by the NWO TOP-C1 project of the Netherlands Organisation for Scientific Research.  The work of Stella Kapodistria is supported by the NWO Gravitation Program NETWORKS of the Dutch government.}


	\bibliographystyle{abbrv}
	\bibliography{Final_version_30th_Bib}


	\begin{appendices}
		\section{An alternative way for the derivation of the stability condition}\label{stab}
		Denote by $A_{r}(n)$ the number of arrivals at relay node $r$, at the beginning of slot $n$ (i.e., $A(n)=A_{1}(n)+A_{2}(n)$, with $\mathbb{E}[A({n})]=\lambda$), and by $S_{r}(n)$ the number of departures from relay node $r$, $r=1,2$, at the end of slot $n$, $n \geq 0$. Then, the queue evolution is as follows
		\begin{displaymath}
		Q_{r}(n+1)=[Q_{r}(n)+A_{r}(n)-S_{r}(n)]^{+}.
		\end{displaymath}
		Equivalently, we can define a non-negative random variable $U_{r}(n)$, which denotes the unused service in a time slot, and rewrite the above equation as
		\begin{displaymath}
		Q_{r}(n+1)=Q_{r}(n)+A_{r}(n)-S_{r}(n)+U_{r}(n),
		\end{displaymath}
		where $U_{r}(n)=\max(0,S_{r}(n)-A_{r}(n)-Q_{r}(n))$.
		\begin{theorem}
			Assume that $\lambda<2a\bar{a}$. Then, $\{\bm{Q}(n),\,n \geq 0\}$ is a positive recurrent Markov chain under the JSRQ policy.
		\end{theorem}
		\begin{proof}
			Consider the Lyapunov function
			\begin{displaymath}
			V(\bm{Q}(n))=Q_{1}(n)^{2}+Q_{1}(n)^{2}.
			\end{displaymath}
			Then,
			\begin{displaymath}
			\begin{array}{rl}
			V(\bm{Q}(n+1))-V(\bm{Q}(n))=&\left((Q_{1}+A_{1}-S_{1}+U_{1})^{2}-Q_{1}^{2}\right)+ \left((Q_{2}+A_{2}-S_{2}+U_{2})^{2}-Q_{2}^{2}\right)\\
			\leq & \left((Q_{1}+A_{1}-S_{1})^{2}-Q_{1}^{2}\right)- \left((Q_{2}+A_{2}-S_{2})^{2}-Q_{2}^{2}\right)\\
			=&(A_{1}-S_{1})^{2}+2Q_{1}(A_{1}-S_{1})+(A_{2}-S_{2})^{2}+2Q_{2}(A_{2}-S_{2}),
			\end{array}
			\end{displaymath}
			where for the sake of readability we suppressed in the notation the $n$. Note that for the system at hand $(A_{r}(n)-S_{r}(n))^{2}\leq 1$, due to the Bernoulli arrivals, and 
			\begin{displaymath}
			\begin{array}{rl}
			\mathbb{E}(S_{1})=&\mathbb{E}(S_{1}|Q_{2}>0)P(Q_{2}>0)+\mathbb{E}(S_{1}|Q_{2}=0)P(Q_{2}=0)\\
			=&a\bar{a}(1-P(Q_{2}=0))+aP(Q_{2}=0)\\
			=&a\bar{a}+a^{2}P(Q_{2}=0).
			\end{array}
			\end{displaymath}
			Note that due to the symmetry of the model $\mathbb{E}(S_{1})=\mathbb{E}(S_{2}):=\mathbb{E}(S)$. Note also that
			\begin{displaymath}
			a\bar{a}\leq \mathbb{E}(S_{r})\leq a,\,r=1,2.
			\end{displaymath}	
			Thus,
			\begin{displaymath}
			\begin{array}{l}
			\mathbb{E}(V(\bm{Q}(n+1))-V(\bm{Q}(n))|\bm{Q}(n)=\bm{Q})  \\\leq 2+2\mathbb{E}(Q_{1}A_{1}+Q_{2}A_{2}|\bm{Q}(n)=\bm{Q})-2[Q_{1}+Q_{2}]\mathbb{E}(S)  \\
			\leq 2-2(Q_{1}+Q_{2})a\bar{a}+2\min[Q_{1},Q_{2}]\mathbb{E}(A(n))\\
			=2-2(Q_{1}+Q_{2})a\bar{a}+2\min[Q_{1},Q_{2}]\lambda\\
			\leq 2-2(Q_{1}+Q_{2})a\bar{a}+2(Q_{1}\frac{a\bar{a}-\epsilon/2}{\lambda}+Q_{2}\frac{\bar{a}a-\epsilon/2}{\lambda})\lambda\\
			=2-\epsilon(Q_{1}+Q_{2}),
			\end{array}
			\end{displaymath}
			for $\epsilon=2a\bar{a}-\lambda>0$.
		\end{proof}
		
		\section{Compensation approach}
		\label{ApendixCA}

		\begin{lemma}\label{lem1s}
			For $\rho=\frac{\lambda(\bar{a}^{2}+a^{2})}{2\bar{\lambda}\bar{a}a}<1$, the balance equations \eqref{in}-\eqref{hor3} have a unique solution of the form
			\begin{equation}\label{eq_hat_p-geom}
			\pi_{k,l}=\rho^{2k}q(l),\,k,l\geq0,
			\end{equation}
			with $q(l)$ non zero such as $\sum_{l=0}^{\infty}\rho^{-2l}q(l)<\infty$.
		\end{lemma}

		\begin{proof}
			We consider a modified model, which is closely related to the $\{\tilde{\bm{Q}}(n),\, n \geq 0\}$ model and it has the same asymptotic behaviour. This modified model corresponds to a random walk on a slightly different grid, namely
			\begin{displaymath}
			\{(k,l):k\geq0,l\geq0\}\cup\{(k,l):k<0,k+l\geq 0\}.
			\end{displaymath}
			In the interior and on the horizontal boundary, the modified model has the same transition rates as the $\{\tilde{\bm{Q}}(n),\, n \geq 0\}$ model. A characteristic feature of the modified model is that its balance equations for $k+l = 0$ are exactly the same as the ones in the interior (in this sense the modified model has no ``vertical'' boundary equations) and both models have the same stability region. Therefore, the balance equations for the modified model are Equations \eqref{in}-\eqref{hor3} for all $k+l\geq0$, $k\in \mathbb{Z}$ with only the equation for state $(0,0)$ being different due to the incoming rates from the states with $k+l=0$, $k\in \mathbb{Z}$.
			
			Observe  that the modified model, restricted to an area of the form $\{(k,l):\, k\geq k_0-l,l\geq 0,k_0=1,2,\ldots\}$ embarked by a line parallel to the $k+l=0$ axis, yields the exact same process. Hence, we can conclude that the equilibrium distribution of the modified model, say $\hat{\pi}_{k,l}$, satisfies
			\begin{equation*}
			\hat{\pi}_{k+1,l}=\gamma \hat{\pi}_{k,l},\ k\geq-l,\ l\geq 0,
			\end{equation*}
			and therefore
			\begin{equation}
			\hat{\pi}_{k,l}=\gamma^k\hat{\pi}_{0,l},\ k\geq-l,\ l\geq 0.
			\label{eq.geometric_p(m,n)_1}
			\end{equation}
			This yields
			\begin{equation*}
			\sum_{l=0}^\infty\hat{\pi}_{-l,l}=\sum_{l=0}^\infty\gamma^{-l}\hat{\pi}_{0,l}<1.
			\end{equation*}
			To determine the $\gamma$ we consider levels of the form $L=\{(k,l)\,:\ 2k+l=L\}$ and let $\hat{\pi}_L=\sum_{2k+l=L}\hat{\pi}_{k,l}$. The balance equations between the levels are given by
			\begin{align}
			\lambda(\bar{a}^2+a^2)\hat{\pi}_L=2\hat{\lambda}\bar{a}a\hat{\pi}_{L+1},\ L\geq1,\label{26}
			\end{align}
			since $p_{1,0}^V=2p_{0,1}^D=\lambda(\bar{a}^2+a^2)$ and $p_{-1,0}^V+p_{0,-1}^V=2p_{-1,0}^D=2\hat{\lambda}\bar{a}a$. Furthermore, Equation \eqref{eq.geometric_p(m,n)_1} yields
			\begin{align}
			\hat{\pi}_{L+1}=\sum_{2k+l=L+1}\gamma^k \hat{\pi}_{0,l}=\gamma\sum_{2k+l=L-1}\gamma^k \hat{\pi}_{0,l}(n)=\gamma\,  \hat{\pi}_{L-1}\label{31}.
			\end{align}
			Substituting \eqref{31} into \eqref{26} yields
			\begin{equation}\label{4.5}
			\gamma=\rho^2=\left(=\frac{\lambda(\bar{a}^{2}+a^{2})}{2\bar{\lambda}\bar{a}a}\right)^2.
			\end{equation}
			So far, we have shown that the equilibrium distribution of the modified model has a product form solution which is unique up to a positive multiplicative constant. Returning to the $\{\tilde{\bm{Q}}(n),\, n \geq 0\}$ model, we can immediately assume that the solution of the balance equations \eqref{in}-\eqref{hor3} is identical to the expression for the modified model as given in \eqref{eq_hat_p-geom}. Furthermore, the above analysis implies that this product form is unique, since the equilibrium distribution of the modified model is unique. \end{proof}

		\begin{lemma}\label{lemma1}
			\begin{enumerate}
				\item[(i)] For a fixed $\gamma$, with $|\gamma|\in(0,1)$, Equation \eqref{eq} has exactly one  $\delta$-root with $0<|\delta|<|\gamma|$.
				\item[(ii)] For a fixed $\delta$, with $|\delta|\in(0,1)$, Equation \eqref{eq} has exactly one  $\gamma$-root with $0<|\gamma|<|\delta|$. 
			\end{enumerate}
		\end{lemma}
		\begin{proof}
			\begin{enumerate}
				\item[(i)]  Divide \eqref{eq} by $\gamma^{2}$ and set $z=\delta/\gamma$. Then after some straightforward algebra, it yields
				\begin{equation*}
				f(z):=z^2\left( \bar{\lambda} \bar{a} a \gamma  + \lambda(a^2 + \bar{a}^2) + \lambda \bar{a} a \gamma z\right) =z\left(1-\left(\bar{\lambda}(\bar{a}^2+a^2) + \lambda a\bar{a}\right) \right)-\bar{\lambda} \bar{a} a:=g(z).
				\end{equation*}
				Note that for $|z|=1$,
				\begin{align*}
				|f(z)|&=|z|^{2}|\bar{\lambda} \bar{a} a \gamma  + \lambda(a^2 + \bar{a}^2) + \lambda \bar{a} a \gamma z|\leq \bar{\lambda} \bar{a} a |\gamma|  + \lambda(a^2 + \bar{a}^2) + \lambda \bar{a} a |\gamma| \\
				&< \bar{\lambda} \bar{a} a + \lambda(a^2 + \bar{a}^2) + \lambda \bar{a} a =\lambda(a^2 + \bar{a}^2)+\bar{a}a,\\
				|g(z)|&\geq \left | |z|\left|1-\left(\bar{\lambda}(\bar{a}^2+a^2) + \lambda a\bar{a}\right) \right|-\bar{\lambda} \bar{a} a\right |=1-\left(\bar{\lambda}(\bar{a}^2+a^2) + \lambda a\bar{a}\right) -\bar{\lambda} \bar{a} a.\\
				&=\lambda(a^2 + \bar{a}^2)+\bar{a}a
				\end{align*}
				Then the lemma follows by applying Rouch\'e's theorem to $f(z)$ and $g(z)$ on the unit circle.
				\item[(ii)] Multiply \eqref{eq} with $\gamma/\delta^{3}$ and set $\hat{z}=\gamma/\delta$. The statement then follows in an analogous manner as in Assertion (i).
			\end{enumerate}\vspace{-.4cm}
		\end{proof}
		
		\begin{proposition}\label{JSQR_prop4}
			The sequences $\{\gamma_i\}_{i\in\mathbb{N}}$ and $\{\delta_i\}_{i\in\mathbb{N}}$ appearing in Equations \eqref{on-final} and \eqref{on2-final} satisfy the properties
			\begin{enumerate}
				\item[(i)] $1>\rho^2=|\gamma_{0}|>|\delta_{0}|>|\gamma_{1}|>|\delta_{1}|>|\gamma_2|>\ldots$,
				\item[(ii)] $0\leq|\gamma_{i}|\leq 0.4^{i}\rho^2$ and $0\leq|\delta_{i}|\leq \frac{1}{2}0.4^{i}\rho^2$.
			\end{enumerate}
		\end{proposition}
		\begin{proof}
			\begin{enumerate}
				\item[(i)] Recall that $\gamma_{0}=\rho^2<1$, then $\delta_{0}$ follows from \eqref{eq} and according to Lemma \ref{lemma1} $|\gamma_{0}|>|\delta_{0}|$. So that assertion (i) follows upon repeating this argument and in light of Lemma \ref{lemma1}. 
				\item[(ii)] We prove the statement of Assertion (ii) by showing that, 
				\begin{enumerate}
					\item  for a fixed $\gamma$,  with $|\gamma|\leq \gamma_0$, $|\delta|<\frac{1}{2}|\gamma|$, and that, 
					\item for a fixed $\delta$, with $|\delta|\leq\gamma_0/2$, $|\gamma|<\frac{8}{10}|\delta|$. 
				\end{enumerate}
				Then, by applying the above iteratively yields
				\begin{align*}
				&|\gamma_i|\leq \frac{8}{10}|\delta_{i-1}|\leq \frac{8}{10}\frac{1}{2}|\gamma_{i-1}|\leq \ldots\leq \left(\frac{8}{10}\frac{1}{2}\right)^i|\gamma_0|=0.4^i\rho^2,\\
				&|\delta_i|\leq \frac{1}{2}|\gamma_i|\leq \frac{8}{10}\frac{1}{2}|\delta_{i-1}|\leq \ldots\leq \left(\frac{8}{10}\frac{1}{2}\right)^i|\delta_0|\leq \frac{1}{2}\left(\frac{8}{10}\frac{1}{2}\right)^i|\gamma_0|=\frac{1}{2}0.4^i\rho^2.
				\end{align*}
				It remains to prove (a) and (b) stated above.
				\begin{enumerate}
					\item  For a fixed $\gamma$,  we prove that $|\delta|<|\gamma|/2$, by repeating the analysis of Lemma \ref{lemma1} for $z=\delta/\gamma$ on the domain $|z|=1/2$, i.e., we show that $|f(z)|<|g(z)|$ on $|z|=1/2$. The domain $|z|=1/2$ is determined by noting that the function $g(z):=z\left(1-\left(\bar{\lambda}(\bar{a}^2+a^2) + \lambda a\bar{a}\right) \right)-\bar{\lambda} \bar{a} a$, defined in the proof of Lemma \ref{lemma1}, has one single root in the interior of the domain $|z|=1/2$.
					
					For $|z|=\frac{1}{2}$,
					\begin{align*}
					|f(z)|&=|z|^{2}\left|\bar{\lambda}a\bar{a}\gamma+\lambda(\bar{a}^{2}+a^{2})+\lambda a\bar{a}\gamma z\right|\leq\frac{1}{4}\left(\bar{\lambda}a\bar{a}|\gamma|+\lambda(\bar{a}^{2}+a^{2})+\frac{\lambda a\bar{a}|\gamma|}{2}\right),
					\end{align*}
					and
					\begin{align*}
					|g(z)|\geq& \left | |z|\left|1-\left(\bar{\lambda}(\bar{a}^2+a^2) + \lambda a\bar{a}\right) \right|-\bar{\lambda} \bar{a} a\right |=\frac{1}{2}\left(1-\left(\bar{\lambda}(\bar{a}^2+a^2) + \lambda a\bar{a}\right)\right)-\bar{\lambda}a\bar{a}\\
					=&\frac{1}{2}\left(\lambda(\bar{a}^2+a^2) + \lambda a\bar{a}+2\bar{\lambda}a\bar{a}\right)-\bar{\lambda}a\bar{a}=\frac{1}{2}\left(\lambda(\bar{a}^2+a^2) + \lambda a\bar{a}\right).
					\end{align*}
					Note that
					\[\frac{1}{2}\left(\lambda(\bar{a}^2+a^2) + \lambda a\bar{a}\right)> \frac{1}{4}\left(\bar{\lambda}a\bar{a}|\gamma|+\lambda(\bar{a}^{2}+a^{2})+\frac{\lambda a\bar{a}|\gamma|}{2}\right)\ \Leftrightarrow\ |\gamma|< \frac{2\lambda}{\bar{a}a(2-\lambda)}\]
					and that
					\[|\gamma|\leq \gamma_0=\rho^2< \frac{2\lambda}{\bar{a}a(2-\lambda)}.\]
					This completes the proof that $|f(z)|<|g(z)|$ on $|z|=1/2$ and the application of Rouch\'e's theorem follows evidently.

					\item For a fixed $\delta$, we prove that $|\gamma|<\frac{8}{10}|\delta|$, using a similar analysis as the one above but now for $\hat{z}=\gamma/\delta$ on the domain $|\hat{z}|=8/10$. To this end, we multiply \eqref{eq} with $\gamma/\delta^{3}$ and set $\hat{z}=\gamma/\delta$, yielding
					\begin{displaymath}
					\hat{f}(\hat{z}):=\hat{z}^{2}\bar{\lambda}a\bar{a}=\hat{z}\left(1-\left(\bar{\lambda}(\bar{a}^{2}+a^{2})+\lambda a\bar{a}\right)-\bar{\lambda}a\bar{a}\delta\right)-\left(\lambda(\bar{a}^{2}+a^{2})+\lambda a\bar{a}\delta\right):=\hat{g}(\hat{z}).
					\end{displaymath}
					The domain $|\hat{z}|=8/10$ is determined by noting that the function $\hat{g}(\hat{z})$ has one single root in the interior of the domain $|\hat{z}|=8/10$, namely that
					\[\left|\frac{\lambda(\bar{a}^{2}+a^{2})+\lambda a\bar{a}\delta}{1-\left(\bar{\lambda}(\bar{a}^{2}+a^{2})+\lambda a\bar{a}\right)-\bar{\lambda}a\bar{a}\delta}\right|
					\leq
					\frac{\lambda(\bar{a}^{2}+a^{2})+\lambda a\bar{a}|\delta|}{\left|1-\left(\bar{\lambda}(\bar{a}^{2}+a^{2})+\lambda a\bar{a}\right)-\bar{\lambda}a\bar{a}|\delta|\right|}\leq \frac{8}{10}.
					\]
					For $|\hat{z}|=\frac{8}{10}$,
					\begin{align*}
					\left|\hat{f}(\hat{z})\right|&=\left(\frac{8}{10}\right)^2\bar{\lambda}a\bar{a},
					\end{align*}
					and
					\begin{align*}
					\left|\hat{g}(\hat{z})\right|\geq& \frac{8}{10}\left|
					\left(1-\left(\bar{\lambda}(\bar{a}^{2}+a^{2})+\lambda a\bar{a}\right)-\bar{\lambda}a\bar{a}|\delta|\right)-\left(\lambda(\bar{a}^{2}+a^{2})+\lambda a\bar{a}|\delta|\right)\right|.
					\end{align*}
					Note that
					\begin{align*}
					\frac{8}{10}\bar{\lambda}a\bar{a}<\left|
					\left(1-\left(\bar{\lambda}(\bar{a}^{2}+a^{2})+\lambda a\bar{a}\right)-\bar{\lambda}a\bar{a}|\delta|\right)-\left(\lambda(\bar{a}^{2}+a^{2})+\lambda a\bar{a}|\delta|\right)\right|\\
					\ \Leftrightarrow\ |\delta|<  \frac{6 a^2 \lambda +24 a^2-6 a \lambda -24 a+5 \lambda }{5 (a-1) a (\lambda +4)}
					\end{align*}
					and that
					\[|\delta|\leq \frac{\gamma_0}{2}=\frac{\rho^2}{2}< \frac{6 a^2 \lambda +24 a^2-6 a \lambda -24 a+5 \lambda }{5 (a-1) a (\lambda +4)}.\]
					This completes the proof that $|\hat{f}(\hat{z})|<|\hat{g}(\hat{z})|$ on $|\hat{z}|=8/10$ and the application of Rouch\'e's theorem follows evidently.
					
				\end{enumerate}
			\end{enumerate}\vspace{-0.8cm}
		\end{proof}

		It follows from Proposition \ref{JSQR_prop4} that the sequences of  $\{\gamma_i\}_{i\in\mathbb{N}}$ and $\{\delta_i\}_{i\in\mathbb{N}}$ tend to zero as $i\to\infty$. The limiting behaviour of  the sequences is presented in  Lemma \ref{lem?1} and the limiting behaviour of the coefficients is treated in Lemma \ref{lem?2}.

		\begin{lemma}\label{lem?1}
			\begin{itemize}
				\item[{\normalfont (i)}]
				Let $\gamma$, $\delta$ and $\hat{\gamma}$ be roots of Equation \eqref{eq} with $1>|\gamma|>|\delta|>|\hat{\gamma}|$. Then, as $\gamma\to 0$,
				${\delta}/{\gamma}\to w$, with $ |w|\in(0,1)$, and ${\hat{\gamma}}/{\delta}\to{1}/{\hat{w}}$, with $ |\hat{w}|>1$, and  $w$ and $\hat{w}$ are the two distinct roots of
				\begin{align}\label{w_i-eq}
				0&=\bar{\lambda} \bar{a} a -w\left( 1-\left(\bar{\lambda}\left(\bar{a}^2+a^2\right) + \lambda a\bar{a}\right) \right)+\lambda\left(a^2 + \bar{a}^2\right)w^2.
				\end{align}
				\item[{\normalfont (ii)}]
				Let $\delta$, $\hat{\gamma}$ and $\hat{\delta}$ be roots of Equation \eqref{eq} with $1>|\delta|>|\hat{\gamma}|>|\hat{\delta}|$. Then, as $\delta\to 0$,
				$\hat{\delta}/\hat{\gamma}\to w$, with $ |w|\in(0,1)$, and ${\hat{\gamma}}/{\delta}\to{1}/{\hat{w}}$, with $ |\hat{w}|>1$, and  $w$ and $\hat{w}$ are the two distinct roots of Equation \eqref{w_i-eq}.
			\end{itemize}
		\end{lemma}
		
		\begin{proof}
			By multiplying \eqref{eq} with $1/\gamma^2$, we obtain after some straightforward calculations
			\begin{align*}
			\frac{\delta}{\gamma}=& \left(\bar{\lambda}(\bar{a}^2+a^2) + \lambda a\bar{a}\right) \frac{\delta}{\gamma} + \left(\bar{\lambda} \bar{a} a \gamma  + \lambda(a^2 + \bar{a}^2)\right) \left( \frac{\delta}{\gamma}\right)^2+ \lambda \bar{a} a  \delta \left( \frac{\delta}{\gamma}\right)^2 + \bar{\lambda} \bar{a} a .
			\end{align*}
			Taking now the limit as $\gamma\to0$ (which also implies that $\delta\to0$) and $\delta/\gamma\to w$ yields \eqref{w_i-eq}. By applying Rouch\'e's theorem in Equation \eqref{w_i-eq}, it makes it immediately evident that the resulting equation has one root inside and one root outside the unit circle. This completes the proof of assertion (i). The proof of assertion (ii) follows similarly. 
		\end{proof}

		\begin{lemma}\label{lem?2}
			\begin{enumerate} 
				\item[{\normalfont (i)}] Let $\gamma$, $\delta$  and $\hat{\gamma}$ be roots of Equation \eqref{eq} for fixed $\delta$, with  $1>|\gamma|>|\delta|>|\hat{\gamma}|$. Then, as $\gamma\to 0$,
				${\delta}/{\gamma}\to w$,
				\begin{align*}
				\frac{c_{i+1}}{d_i}&\to -\frac{1-\bar{\lambda}\bar{a}-\lambda a\bar{a}-\bar{\lambda}a\bar{a}/w} {1-\bar{\lambda}\bar{a}-\lambda a\bar{a}-\bar{\lambda}a\bar{a}/\hat{w}}.
				\end{align*}
				\item[{\normalfont (ii)}] Let $\delta$, $\hat{\gamma}$ and $\hat{\delta}$ be roots of Equation \eqref{eq} for fixed $\hat{\gamma}$, with  $1>|\delta|>|\hat{\gamma}|>|\hat{\delta}|$. Then, as $\delta\to 0$,
				$\hat{\delta}/\hat{\gamma}\to w$,
				\begin{align*}
				\frac{e_{0,i}}{d_i \delta_i}&\to \frac{\left(2 a^2-2 a+1\right) \lambda  w (\hat{w}-w)}{-4 a^2 \lambda +2 a^2 \lambda  \hat{w}+2 a^2+4 a \lambda -2 a \lambda  \hat{w}-2 a-\lambda +\lambda  \hat{w}},
				\\
				\frac{e_{1,i}}{d_i \delta_i}&\to \frac{\left(4 a^2 \lambda -2 a^2-4 a \lambda +2 a+\lambda \right) (\hat{w}-w)}{-4 a^2 \lambda +2 a^2 \lambda  \hat{w}+2 a^2+4 a \lambda -2 a \lambda  \hat{w}-2 a-\lambda +\lambda  \hat{w}},
				\\
				\frac{c_i}{d_i}&\to -\frac{w^2 \left(-4 a^2 \lambda +2 a^2 \lambda  w+2 a^2+4 a \lambda -2 a \lambda  w-2 a-\lambda +\lambda  w\right)}{\hat{w}^2 \left(-4 a^2 \lambda +2 a^2 \lambda  \hat{w}+2 a^2+4 a \lambda -2 a \lambda  \hat{w}-2 a-\lambda +\lambda  \hat{w}\right)}.
				\end{align*}
			\end{enumerate}
		\end{lemma}
		
		\begin{proof}
			The proof of the two assertions follows evidently by applying the results of Lemma \ref{lem?1} to Equations \eqref{JSRQ_C} and \eqref{System1}, respectively. 
		\end{proof}

		The following proposition states that the series \eqref{on} and \eqref{on2} converges absolutely.
		\begin{proposition}\label{prop_Abs_Conv}
			There exists a positive integer $N$ such that
			\begin{enumerate}
				\item[{\normalfont (i)}]  The series $\sum_{i=0}^{\infty}(d_{i}\gamma_{i}^{k}+c_{i+1}\gamma_{i+1}^k)\delta_{i}^{l}$, for $k+l>N$, converges absolutely.
				\item[{\normalfont (ii)}]  The series $\sum_{i=1}^{\infty}e_{l,i}\gamma_{i}^{k}\delta_0^l$, for $k>N$, $l=0,1$, converges absolutely.
				\item[{\normalfont (iii)}]   The series
				$\sum_{k+l>N}\sum_{i=0}^{\infty}(d_{i}\gamma_{i}^{k}+c_{i+1}\gamma_{i+1}^k)\delta_{i}^{l} +\sum_{k>N}\sum_{l=0}^1\sum_{i=1}^{\infty}e_{l,i}\gamma_{i}^{k}\delta_0^l$
				converges absolutely.
			\end{enumerate}
		\end{proposition}
		
		\begin{proof}
			It suffices to show that the series $\sum_{i=0}^{\infty}d_{i}\gamma_{i}^{k}\delta_{i}^{l}$ and $\sum_{i=0}^{\infty}c_{i+1}\gamma_{i+1}^k\delta_{i}^{l}$ converge absolutely. The rest follows along the same lines as in \cite{ad0}. Note that, from Lemmas \ref{lem?1} and \ref{lem?2}, 
			\begin{align}\label{JSRQ_RHS_AbsConv}
			\left|\frac{d_{i+1}\gamma_{i+1}^{k}\delta_{i+1}^{l}}{d_{i}\gamma_{i}^{k}\delta_{i}^{l}}\right|&\to
			\frac{\left|\frac{1-\bar{\lambda}\bar{a}-\lambda a\bar{a}-\bar{\lambda}a\bar{a}/\hat{w}}{1-\bar{\lambda}\bar{a}-\lambda a\bar{a}-\bar{\lambda}a\bar{a}/w}\right|
			}{
				\left|\frac{\hat{w}^2 \left(-4 a^2 \lambda +2 a^2 \lambda  \hat{w}+2 a^2+4 a \lambda -2 a \lambda  \hat{w}-2 a-\lambda +\lambda  \hat{w}\right)}{w^2 \left(-4 a^2 \lambda +2 a^2 \lambda  w+2 a^2+4 a \lambda -2 a \lambda  w-2 a-\lambda +\lambda  w\right)}\right|
			}\left|\frac{w}{\hat{w}}\right|^{k+l}.
			\end{align}
			As $|w|,1/|\hat{w}|<1$, then there exists a positive constant $N_1$, such that for $k+l>N_1$ the right hand side of Equation \eqref{JSRQ_RHS_AbsConv} is bounded by some positive constant smaller than one. Thus, the series $\sum_{i=0}^{\infty}d_{i}\gamma_{i}^{k}\delta_{i}^{l}$ is bounded by a geometric series for $k+l>N_1$, and as such converges absolutely. 
			This concludes the proof of the absolute convergence of the series $\sum_{i=0}^{\infty}d_{i}\gamma_{i}^{k}\delta_{i}^{l}$.
			Similarly, one can show that the series  $\sum_{i=0}^{\infty}c_{i+1}\gamma_{i+1}^k\delta_{i}^{l}$ converges absolutely for $k+l>N_2$, with $N_2$ some positive constant not necessarily equal to $N_1$. 
		\end{proof}

		\section{Power series algorithm implementation}
		\label{ApendixPSA}
		In this section, we explain the power series implementation for $a = 1/2$ and $\lambda = \frac{\rho}{1 + \rho}$. Now substituting $a$ and $\lambda$ into \eqref{RB1}-\eqref{RB3}-\eqref{hor1}, we express the balance equations in term of the parameter $\rho$
		%
		
		\begin{align}
		\pi_{0, 1} &= \rho \pi_{0, 0}, \label{SB1} \\
		\pi_{0, 2} &=    (1 + \rho) \pi_{0, 1} - \pi_{1, 0}  - \rho \pi_{0, 0}, \label{SB2}\\
		2 \pi_{0, l+1}&=   (2+ 3 \rho) \pi_{0, l} - \pi_{1, l - 1} - \rho \pi_{1, 0} \mathds{1}_{(l = 2)}, \ l \geq 2, \label{SB4} \\
		(2 + 3 \rho)  \pi_{k, 0}  &=   \pi_{k, 1}   + 2 \rho \pi_{k - 1, 1} + \rho \pi_{k - 1, 2}, \ k \geq 1, l = 0, \label{SB5}\\
		2(1 + \rho)  \pi_{k, 1} &=  \pi_{k, 2}  +  2 \rho \pi_{k - 1, 2} + \rho \pi_{k - 1, 3} + 2 \rho \pi_{k, 0} + 2 \pi_{k + 1, 0},  \ k \geq 1, l = 1, \label{SB6} \\
		(2 + 3 \rho)  \pi_{k, l}  &=  \pi_{k, l+1} + 2 \rho \pi_{k - 1, l+1} + \rho \pi_{k - 1, l+2} \nonumber\\
		& + \pi_{k+1, l-1} + \rho \pi_{k + 1, 0} \mathds{1}_{(l = 2)},  \ k \geq 1, l \geq 2. \label{SB8} 
		\end{align}	
		
		Substituting $ \pi_{k, l} $ (defined in \eqref{PGF}) into \eqref{SB1}-\eqref{SB8} and equating the powers of $\rho$  leads to the following recursion equations for the coefficients, for $n = 0, 1, \dots$,

		\begin{align}
		\beta(n-1, 0, 1) &= \beta(n-1, 0, 0), \  n \geq 1, ~ l, k =0, ~\label{CB1} \\
		\beta(n-2, 0, 2)  &= \beta(n-1, 0, 1) + \beta(n-2, 0, 1) \nonumber \\ 
		&\ -  \beta(n-1, 1, 0)  -\beta(n-1, 0, 0), \ n \geq 2, l = 1, k = 0,~\label{CB2} \\
		\beta(n-l-1, 0, l+1)  &=   \beta(n-l, 0, l) + \frac{3}{2} \beta(n-l-1, 0, l)  - \frac{1}{2} \beta(n-l, 1, l-1) \nonumber \\ 
		&\ - \frac{1}{2} \beta(n-l, 1, 0)  \mathds{1}_{\{l =2\}}, \   n \geq l+1, l \geq 2, k =0,  \label{CB3} \\
		\beta(n - k, k, 0)  &=   \left[ - \frac{3}{2} \beta(n - k-1, k, 0)  + \frac{1}{2} \beta(n - k - 1, k, 1) \right]\mathds{1}_{\{n > k\}} \\ 
		&\ + \left[ \beta(n -k-1, k-1, 1) \right]\mathds{1}_{\{n > k\}}  \nonumber \\ 
		&\ +   \frac{1}{2} \beta(n - k-2, k - 1, 2) \mathds{1}_{\{n > k + 1\}}, ~ n \geq k,~ k\geq 1, \label{CB4} \\
		\beta(n-k-1, k, 1)  &=  \left[ - \beta(n - k-2, k, 1) + \frac{1}{2} \beta(n - k-2, k, 2) \right] \mathds{1}_{\{n > k+1\}} \nonumber \\ 
		&\ +    \left[ \beta(n -k-2, k-1, 2) \right] \mathds{1}_{\{n > k+1\}} \nonumber \\ 
		&\ +   \frac{1}{2} \beta(n - k-3, k - 1, 3) \mathds{1}_{\{n > k+2\}} +  \beta(n-k-1, k, 0) \nonumber \\ 
		&\ +   \beta(n-k-1, k + 1, 0) ,  \  n \geq k+1, k \geq 1, \label{CB5} \\
		\beta(n-k-l, k, l)  &=   \left[- \frac{3}{2} \beta(n - k-l-1, k, l)  + \frac{1}{2} \beta(n - k-l-1, k, l+1) \right]\mathds{1}_{\{n > k+l\}} \nonumber \\ 
		&\ +   \beta(n - k-l-1, k-1, l+1) \mathds{1}_{\{n > k+l\}} + \frac{1}{2} \beta(n-k-l, k + 1, l - 1) \nonumber \\ 
		&\ + \frac{1}{2} \beta(n - k-l-2, k - 1, l+2)\mathds{1}_{\{n > k+l+1\}} \nonumber \\ 
		&\ + \frac{1}{2} \beta(n-k-l, k + 1, 0)  \mathds{1}_{\{l = 2\}}, \ n \geq k + l,  k \geq 1, l \geq 2. \label{CB6} 
		\end{align}
		
		Furthermore, we use the normalisation equation to derive one more equation required for the determination of $\beta(n, 0, 0)$, $n = 0, 1,\ldots$. To this purpose, we use the law of total probability, see \cite[p.~161]{blanc1992}.
Substituting \eqref{PGF} and equating the corresponding powers of $\rho$ yields
		\begin{equation}
		\beta(0, 0, 0) = 1, \, \, \, \beta(n, 0, 0) = -\mathop{\sum\sum}\limits_{0 < k + l \leq n} \beta(n - k - l, k, l), \ n=1,2,\dots. \label{CB7} 
		\end{equation}
		The coefficients $\beta(n, k, l)$ can be recursively calculated from \eqref{CB1}-\eqref{CB7}. Thus, using \eqref{PGF}, one can determine the equilibrium distribution $\pi_{k, l}$.

		Similarly, for \eqref{PGF_new1}	, we repeat the above computations so as to determine the coefficient $u(n, k, l), \ n \geq 0,\ (k,l)\in\hat{S}$, 
		\begin{align}
		u(n-1, 0, 1) &= \frac{G}{G+1}u(n-2, 0, 1) \mathds{1}_{\{n \geq 2\}} +  \frac{1}{G+1} u(n-1, 0, 0),  \ n \geq 1,  ~\label{CBN1} \\
		u(n-2, 0, 2)  &= \frac{G}{G+1}u(n-3, 0, 2) \mathds{1}_{\{n \geq 3\}} -  \frac{G-1}{G+1} u(n-2, 0, 1) \nonumber \\ 
		&\ +   u(n-1, 0, 1)  - u(n-1, 1, 0) + \frac{G}{G+1} u(n-2, 1, 0)   \nonumber \\  
		&\ - \frac{1}{G+1}  u(n-1, 0, 0), \ n \geq 2, \label{CBN2} \\
		u(n-l-1, 0, l+1)  &=  \frac{G}{G+1} u(n-l-2, 0, l+1) \mathds{1}_{\{n \geq l+ 2\}} - \frac{G - \frac{3}{2}}{G+1}  u(n-l-1, 0, l) \\ 
		&\ + u(n-l, 0, l)  + \f{1}{2} \frac{G}{G+1}   u(n-l-1, 1, l-1)  - \f{1}{2} u(n-l, 1, l-1) \nonumber \\ 
		&\ -  \frac{1}{2(G+1)}  u(n-l, 1, 0) \mathds{1}_{\{l = 2\}},    \ n \geq l-1, l \geq 2,\label{CBN3} \\
		u(n-k, k, 0)  &= \left(  \frac{G - \frac{3}{2}}{G+1} u(n-k-1, k, 0)  + \frac{1}{2}  u(n-k-1, k, 1)  \right) \mathds{1}_{\{n \geq k+1\}} \nonumber \\ 
		&\ - \frac{1}{2}  \frac{1}{G+1} \left( G  u(n-k-2, k, 1)  -   u(n-k-2, k-1, 2) \right) \mathds{1}_{\{n \geq k+2\}}  \nonumber \\ 
		&\ + \frac{1}{G+1}  u(n-k-1, k-1, 1)   \mathds{1}_{\{n \geq k+1\}}, \ n \geq k, k \geq 1,\label{CBN4} \\
		u(n-k-1, k, 1)  &= \left(  \frac{G - 1}{G+1}   u(n-k-2, k, 1)  + \frac{1}{2} u(n-k-2, k, 2)  \right) \mathds{1}_{\{n \geq k+2\}} \nonumber \\ 
		&\ + \frac{1}{2}  \frac{1}{G+1} \left( - G  u(n-k-3, k, 2) +   u(n-k-3, k-1, 3) \right) \mathds{1}_{\{n \geq k+3\}}  \nonumber \\ 
		&\ + \frac{1}{G+1} \left[   u(n-k-2, k-1, 2) -   G u(n-k-2, k+1, 0) \right] \mathds{1}_{\{n \geq k+2\}}  \nonumber \\ 
		&\ +\frac{1}{G+1}  u(n-k-1, k, 0)  \nonumber \\ 
		&\ +  u(n-k-1, k+1, 0) , \ n \geq k+1, k \geq 1,\label{CBN5} \\
		u(n-k-l, k, l)  &= \left(  \frac{G - \frac{3}{2}}{G+1}  u(n-k-l-1, k, l)  + \frac{1}{2} u(n-k-l-1, k, l+1)  \right) \mathds{1}_{\{n \geq k+l+1\}} \nonumber \\ 
		&\ - \frac{1}{2}  \frac{G}{G+1}   u(n-k-l-2, k, l+1)    \mathds{1}_{\{n \geq k+l+2\}} \nonumber \\ 
		&\ + \frac{1}{2}  \frac{1}{G+1} \left[  u(n-k-l-2, k-1, l+2) \right] \mathds{1}_{\{n \geq k+l+2\}}  \nonumber \\ 
		&\ + \frac{1}{G+1}    u(n-k-l-1, k-1, l+1)   \mathds{1}_{\{n \geq k+l+1\}} \nonumber \\ 
		&\ - \f{1}{2} \frac{G}{G+1}  u(n-k-l-1, k+1, l-1)  \mathds{1}_{\{n \geq k+l+1\}}  \nonumber \\ 
		&\ - \f{1}{2}  u(n-k-l, k+1, l-1) \nonumber \\
		& \ + \frac{1}{2} \frac{1}{G+1} u(n-k-2, k+1, 0) \mathds{1}_{\{l = 2\}}, \  n \geq k+l, k \geq 1, l \geq 2.\label{CBN6}
		\end{align}
		Plus, 	
		\begin{equation}
		u(0, 0, 0) = 1, \, \, \, u(n, 0, 0) = -\mathop{\sum\sum}\limits_{0 < k + l \leq n} u(n - k - l, k, l). \label{CBN7} 
		\end{equation}
		The coefficients $u(n, k, l)$ can be recursively calculated from \eqref{CBN1}-\eqref{CBN7}. Thus, using \eqref{PGF_new1}, one can determine the equilibrium distribution $\pi_{k, l}$.

		\section{Generating function approach}	
		\begin{lemma}\label{lemgf1}
			The equation \begin{equation}
			\begin{array}{l}
			T(r_{1},r_{2})=0\Leftrightarrow r_{1}^{2}[\lambda(\bar{a}^{2}+a^{2})+\frac{\lambda\bar{a}a}{r_{2}}]-r_{1}[\lambda(\bar{a}^{2}+a^{2})+a\bar{a}+\bar{\lambda}\bar{a}a(1-\frac{1}{r_{2}})]+\bar{\lambda}a\bar{a}=0.
			\end{array}
			\label{mk}
			\end{equation} has two roots, say $\delta_{0}(r_{2})$, $\delta_{1}(r_{2})$, such that for $|r_{2}|=1$, $r_{2}\neq1$, $|\delta_{0}(r_{2})|<1<|\delta_{1}(r_{2})|$. For $r_{2}=1$, $\delta_{0}(1)=min[1,\frac{\bar{\lambda}a\bar{a}}{\lambda(\bar{a}^{2}+a^{2})+\lambda\bar{a}}]$, and $\delta_{1}(1)=max[1,\frac{\bar{\lambda}a\bar{a}}{\lambda(\bar{a}^{2}+a^{2})+\lambda\bar{a}a}]$.
		\end{lemma}
		\begin{proof}
			Using Rouch\'e's theorem we can show that \eqref{mk} has a unique zero, say $\delta_{0}(r_{2})$ in $|r_{1}|\leq 1$ if $\mathcal{R}\mathit{e}(1-\frac{1}{r_{2}})\geq 0$. Indeed, for $|r_{1}|=1$, let $b(r_{1})=-r_{1}[\lambda(\bar{a}^{2}+a^{2})+\bar{\lambda}a\bar{a}+\lambda\bar{a}a+\bar{\lambda}\bar{a}a(1-\frac{1}{r_{2}})]$ and $a(r_{1})=r_{1}^{2}[\lambda(\bar{a}^{2}+a^{2})+\frac{\lambda\bar{a}a}{r_{2}}]+\bar{\lambda}a\bar{a}$. Then,
			\begin{displaymath}
			\begin{array}{rl}
			|b(r_{1})|=&|\lambda(\bar{a}^{2}+a^{2})+\bar{\lambda}a\bar{a}+\lambda\bar{a}a+\bar{\lambda}\bar{a}a(1-\frac{1}{r_{2}})|\\
			=&|\lambda(\bar{a}^{2}+a^{2})+\bar{\lambda}a\bar{a}+\frac{\lambda\bar{a}a}{r_{2}}+\bar{a}a(1-\frac{1}{r_{2}})|\\
			\geq &|\lambda(\bar{a}^{2}+a^{2})+\bar{\lambda}a\bar{a}+\frac{\lambda\bar{a}a}{r_{2}}|
			\geq|r_{1}^{2}(\lambda(\bar{a}^{2}+a^{2})+\frac{\lambda\bar{a}a}{r_{2}})+\bar{\lambda}a\bar{a}|=|a(r_{1})|,
			\end{array}
			\end{displaymath}
			where the first inequality stands for $\mathcal{R}\mathit{e}(1-\frac{1}{r_{2}})\geq 0$. Thus, $\delta_{0}(r_{2})$ is the unique zero of $T(r_{1},r_{2})=0$ in $|r_{1}|\leq1$ if $\mathcal{R}\mathit{e}(1-\frac{1}{r_{2}})>0$, and $|\delta_{0}(r_{2})|<1$, for $\mathcal{R}\mathit{e}(|1-\frac{1}{r_{2}}|)\geq 0$, $r_{2}\neq 1$. The other zero, say $\delta_{1}(r_{2})$ lies outside the unit disk, i.e., for $|r_{2}|=1$, $|\delta_{0}(r_{2})|<1<|\delta_{1}(r_{2})|$. For $r_{2}=1$, $T(r_{1},1)=0$, implies $(1-r_{1})[\lambda(\bar{a}^{2}+a^{2})+\lambda\bar{a}a-\frac{\bar{\lambda}a\bar{a}}{r_{1}}]=0$. Therefore, for $r_{2}=1$, $\delta_{0}(1)=min[1,\frac{\bar{\lambda}a\bar{a}}{\lambda(\bar{a}^{2}+a^{2})+\lambda\bar{a}a}]$, and $\delta_{1}(1)=max[1,\frac{\bar{\lambda}a\bar{a}}{\lambda(\bar{a}^{2}+a^{2})+\lambda\bar{a}a}]$.
			
			Similar results hold for $T(r_{2},r_{1})=0$. In particular, $T(r_{2},r_{1})=0$, has a unique zero, say $\delta_{0}(r_{1})$ in $|r_{2}|\leq 1$ for $\mathcal{R}\mathit{e}(1-\frac{1}{r_{1}})>0$, and for $\mathcal{R}\mathit{e}(|1-\frac{1}{r_{1}}|)\geq 0$, $r_{1}\neq 1$. The other zero, say $\delta_{1}(r_{1})$, is such that $|\delta_{1}(r_{1})|>1$. 
		\end{proof}
		\begin{lemma}\label{lemgf2}
			For $r_{2}\in[r_{2}^{(0)},r_{2}^{(1)}]$, the two-valued function $\delta_{1}(r_{2})$ lies on a closed contour
			$\mathcal{M}$, which is symmetric with respect to the
			real line and defined by
			\begin{displaymath}
			\begin{array}{c}
			|\delta_{1}|^{2}=m_{1}(\mathcal{R}\mathit{e}(\delta_{1})),\,\,m_{1}(\zeta)=\frac{\bar{\lambda}a\bar{a}}{\lambda(\bar{a}^{2}+a^{2})+\frac{\lambda\bar{a}a}{l(\zeta)}},\,\,|\delta_{1}|^{2}\leq\frac{\bar{\lambda}a\bar{a}}{\lambda(\bar{a}^{2}+a^{2})+\frac{\lambda\bar{a}a}{r_{2}^{(1)}}},
			\end{array}
			\end{displaymath}
			where
			\begin{displaymath}
			l(\zeta)=\frac{\bar{\lambda}\bar{a}a+2\lambda\bar{a}a\zeta}{\lambda(1-a\bar{a})+2\bar{\lambda}a\bar{a}-2\lambda(\bar{a}^{2}+a^{2})\zeta}.
			\end{displaymath}
			Set $\beta_{0}:=\sqrt{\frac{\bar{\lambda}a\bar{a}}{\lambda(\bar{a}^{2}+a^{2})+\frac{\lambda\bar{a}a}{r_{2}^{(1)}}}}$ the extreme right point of $\mathcal{M}$. Exactly the same result holds for $r_{1}\in[r_{1}^{(0)},r_{1}^{(1)}]$, and $\delta_{0}(r_{1})$ lies on a closed contour $\mathcal{M}$.
		\end{lemma}
		\begin{proof}
			For $r_{2}\in[r_{2}^{(0)},r_{2}^{(1)}]$, $\delta_{0}(r_{2})$, $\delta_{1}(r_{2})$ are complex conjugates, so from \eqref{mk}, 
			\[|\delta_{1}(r_{2})|^{2}=\frac{\bar{\lambda}a\bar{a}}{\lambda(\bar{a}^{2}+a^{2})+\frac{\lambda\bar{a}a}{r_{2}}}.\]
			Clearly, $|\delta_{1}(r_{2})|^{2}$ is an increasing function in $r_2\in[0,1]$, and thus, $|\delta_{1}(r_{2})|^{2}\leq |\delta_{1}(r_{2}^{(1)})|^{2}$. Finally, $l(\zeta)$ is obtained by solving $\mathcal{R}\mathit{e}(\delta_{1}(r_{2}))=\frac{\lambda(\bar{a}^{2}+a^{2})+\bar{\lambda}a\bar{a}+\lambda\bar{a}a+\bar{\lambda}\bar{a}a(1-\frac{1}{r_{2}})}{2[\lambda(\bar{a}^{2}+a^{2})+\frac{\lambda\bar{a}a}{r_{2}}]}$, for $r_{2}$, with $\zeta=\mathcal{R}\mathit{e}(\delta_{1}(r_{2}))$.
		\end{proof}
		\begin{lemma}\label{lemgf3}
			$\delta_{0}(r_{2})$ is analytic in $\mathcal{C}-[r_{2}^{(0)},r_{2}^{(1)}]$.
		\end{lemma}
		\begin{proof}
			Using \eqref{mk} denote
			\begin{displaymath}
			\begin{array}{rl}
			s(r_{2})&=r_{2}[\lambda(1-a\bar{a})+\bar{\lambda}(\bar{a}a+a\bar{a})]-\bar{\lambda}\bar{a}a,\\
			D_{1}(r_{2})&=s^{2}(r_{2})-4\bar{\lambda}a\bar{a}r_{2}(\lambda(\bar{a}^{2}+a^{2})r_{2}+\lambda\bar{a}a).
			\end{array}
			\end{displaymath}
			Note that for $r_{2}\in(-\infty,+\infty)-[r_{2}^{(0)},r_{2}^{(1)}]$, $s(r_{2})\neq 0$. Indeed $s(r_{2})=0$ for $r_{2}^{*}=\frac{\bar{\lambda}\bar{a}a}{\lambda(1-a\bar{a})+\bar{\lambda}(\bar{a}a+a\bar{a})}$. It is easy to realize that $0<r_{2}^{*}<1$. Moreover, it is readily shown that $D_{1}(0)>0$, $D_{1}(r_{2}^{*})=-4\bar{\lambda}a\bar{a}r_{2}^{*}(\lambda(\bar{a}^{2}+a^{2})r_{2}^{*}+\lambda\bar{a}a)<0$, $D_{1}(1)=[\lambda(\bar{a}^{2}+a^{2}+\bar{a}a)-\bar{\lambda}a\bar{a}]^{2}\geq 0$, and by definition $D_{1}(r_{2}^{(0)})=0=D_{1}(r_{2}^{(1)})$. Thus, $r_{2}^{(0)}<r_{2}^{*}<r_{2}^{(1)}$. Then, for $r_{2}\in(-\infty,+\infty)-[r_{2}^{(0)},r_{2}^{(1)}]$, 
			\begin{equation}
			\begin{array}{rl}
			\delta_{0}(r_{2})=&\frac{2\bar{\lambda}a\bar{a}r_{2}}{-s(r_{2})+\sqrt{D_{1}(r_{2})}},\,if\,-s(r_{2})>0,\\
			\delta_{0}(r_{2})=&\frac{2\bar{\lambda}a\bar{a}r_{2}}{-s(r_{2})-\sqrt{D_{1}(r_{2})}},\,if\,-s(r_{2})<0,\\
			\delta_{0}(r_{2})\delta_{1}(r_{2})=&\frac{\bar{\lambda}a\bar{a}r_{2}}{\lambda(\bar{a}^{2}+a^{2})r_{2}+\lambda\bar{a}a}.
			\end{array}
			\label{mlp}
			\end{equation}
			
			Note that \eqref{mlp} implies that $\delta_{0}(r_{2})$ has no poles and vanishes at $r_{2}=0$. Moreover $\delta_{1}(r_{2})$ has no zeros and one pole equal to $\bar{r}_{2}=-\frac{\lambda\bar{a}a}{\lambda(\bar{a}^{2}+a^{2})}$. Similar results hold for $\delta_{0}(r_{1})$, $\delta_{1}(r_{1})$.
		\end{proof}
		
	\end{appendices}

\end{document}